\documentclass[10pt]{amsart}

\usepackage{amsmath,amsthm,amssymb,mathrsfs,amsfonts,verbatim,enumitem,color,leftidx}
\usepackage{units}
\usepackage{mathabx}
\usepackage{etoolbox} 
\usepackage{bbm}
\usepackage[all,tips]{xy}
\usepackage{graphicx,ifpdf}
\ifpdf
\fi

\usepackage[usenames,dvipsnames]{pstricks}
 \usepackage{epsfig}
 \usepackage{pst-grad} 
 \usepackage{pst-plot} 

\newtheorem{thm}{Theorem}[section]
\newtheorem{lem}[thm]{Lemma}
\newtheorem{cor}[thm]{Corollary}
\newtheorem{prop}[thm]{Proposition}

\theoremstyle{definition}

\theoremstyle{remark}
\newtheorem{rem}[thm]{Remark}

\numberwithin{equation}{section}

\makeatletter
\newcommand{\rmnum}[1]{\romannumeral #1}
\newcommand{\Rmnum}[1]{\expandafter\@slowromancap\romannumeral #1@}
\makeatother

\newcommand{\be}{\mathbf{e}}

\newcommand{\br}{\mathbf{r}}

\newcommand{\bu}{\mathbf{u}}
\newcommand{\bv}{\mathbf{v}}
\newcommand{\bw}{\mathbf{w}}
\newcommand{\bx}{\mathbf{x}}

\newcommand{\R}{\mathbb{R}}
\newcommand{\N}{\mathbb{N}}

\newcommand{\Z}{\mathbb{Z}}
\newcommand{\Q}{\mathbb{Q}}

\newcommand{\til}{\widetilde}

\newcommand{\dd}{\; \mathrm{d}}

\newcommand{\cov}{{\mathrm {cov}}}
\newcommand{\vol}{{\mathrm {vol}}}
\newcommand{\DI}{\mathrm{DI}}

\newcommand{\veps}{\varepsilon}

\begin{document}

\title{Hausdorff dimension of   weighted singular vectors in $\R^2$}


\author{Lingmin Liao}
\address{LAMA UMR 8050 CNRS, Universit\'e Paris-Est Cr\'eteil,
61 Avenue du G\'en\'eral de Gaulle, 94010 Cr\'eteil Cedex, France}
\email{lingmin.liao@u-pec.fr}

\author{Ronggang Shi}
\address{Shanghai Center for Mathematical Sciences, Fudan University, Shanghai 200433, PR China}
\email{ronggang@fudan.edu.cn}
\thanks{}

\author{Omri N. Solan}
\address{School of Mathematical Sciences, Tel Aviv University, Tel Aviv 69978, Israel}
\email{omrisola@mail.tau.ac.il}

\author{Nattalie Tamam}
\address{School of Mathematical Sciences, Tel Aviv University, Tel Aviv 69978, Israel}
\email{natalita@post.tau.ac.il}

\subjclass[2010]{Primary   11J13; Secondary 11K55, 37A17.}

\date{}


\keywords{Diophantine approximation,  weighted singular vector, Hausdorff dimension, homogeneous dynamics}

\begin{abstract}
Let $w=(w_1, w_2)$ be a pair of positive  real numbers with $w_1+w_2=1$ and $w_1\ge w_2$. 	
We show that the set of $w$-weighted singular vectors in $\R^2$ has Hausdorff dimension 
$2- \frac{1}{1+w_1}$. This extends the previous work of Yitwah Cheung on 
 the Hausdorff dimension of the usual  (unweighted) singular vectors in $\R^2$. 
\end{abstract}

\maketitle

\markright{}
\markleft{}

\section{Introduction}

 Let  $w=(w_1,w_2)$    
be a pair of positive  real numbers such that $w_1+w_2=1$.
Dirichlet's theorem with weight $w$
 (see \cite[Chapter \Rmnum{2}]{schmidt})\footnote{The form we state below is not explicitly stated 
 	in \cite{schmidt}, but it is a special case of    \cite[Theorem \Rmnum{2}.2C]{schmidt}.}
 states that for all $x=(x_1, x_2)\in \R^2$ and  $T> 1$ there is 
$(p, q)=(p_1, p_2, q)\in \Z^2\times \Z$ such that 
\begin{align*}
\left\{ 
\begin{array}{rl}
|qx_1-p_1|\hskip -6pt &< T^{-w_1} \\
|qx_2-p_2| \hskip -6pt  &< T^{-w_2} \\
0<q \hskip -6pt &\le  T
\end{array}
\right.
.
\end{align*}
There are different classes of vectors in $\R^2$ with more elaborate
Diophantine properties, e.g.~badly approximable vectors, Dirichlet's improvable 
vectors and singular vectors. 
Usually these sets have zero Lebesgue measure. 
The estimation of the size of them 
has a long history and is still  fast developing  in recent years, see 
e.g.~\cite{a1, a2, bpv, b15, c11, cc}.

A vector $x=(x_1, x_2)\in \R^2$ is said to be   $w$-singular if for every $\varepsilon>0$ there
exists $T_0>1$ such that for all $T>T_0$ the system of inequalities
\begin{align}\label{eq;improve}
\left\{ 
\begin{array}{rl}
|qx_1-p_1| \hskip -6pt &<\varepsilon^{w_1} T^{-w_1} \\
|qx_2-p_2| \hskip -6pt &<\varepsilon^{w_2} T^{-w_2} \\
0<q \hskip -6pt &< T
\end{array}
\right.
\end{align}
admits an integer solution $(p,q)\in \Z^2\times \Z$.
The set of $w$-singular vectors is denoted by  $\mathrm{Sing}(w)$. 
A vector $x\in \R^2$ is said to be singular if 
it is 
 $w$-singular  in the case where $w$ is unweighted, i.e.~when $w_1=w_2=\frac{1}{2}$.  

It is proved by Cheung \cite{c11} that the Hausdorff dimension of the set of  singular vectors in $\R^2$ is
$\frac{4}{3}$. Here and hereafter the Hausdorff dimension of a subset of $\R^d$ ($d\in \N:=\{1, 2, \ldots\}$) is with respect 
to the usual Euclidean metric.  Recently, Cheung and Chevallier \cite{cc} extended this result 
to   $\R^d$   $(d\ge 2)$ and proved that the set of singular vectors in $\R^d$ has Hausdorff 
dimension $\frac{d^2}{d+1}$. Recall that in  $\R$ only rational numbers are singular, so 
we understand the Hausdorff dimension of the set of  singular vectors in all Euclidean spaces.

The aim of this paper is to calculate  the Hausdorff dimension of the set  
of $w$-singular vectors in $\mathbb{R}^2$. 
\begin{thm}\label{thm;main}
	Suppose $w=(w_1, w_2)$ where  $w_1\ge w_2>0$ and $w_1+w_2=1$. Then  the Hausdorff dimension of $\mathrm{Sing}(w)$
	is 
	$2- \frac{1}{1+w_1}$.
\end{thm} 

\begin{rem}
 In the case where $w=(1,0)$ one can also define $w$-singular vectors in a similar way. The above formula  
 does not  hold in this degenerate case where the Hausdorff dimension is $1$. By symmetry, we can draw the whole picture of the Hausdorff dimension of $\mathrm{Sing}(w)$ when $w_1$ goes from $0$ to $1$.  We point out that the dimension graph has a non differential point $1/2$ and has jumps at $0$ and $1$ (see Figure 1). 
\end{rem}

\begin{figure}
\centering
\begin{pspicture}(1,-0.5)(1.25,5)
\psset{unit=2cm}
\psaxes{->}(0,0)(-0.7,-0.25)(1.9,2.25)
\psplot[plotstyle=curve, linewidth=0.035cm]{0.5}{1}{x x add 1 add x 1 add div }
\psplot[plotstyle=line,linestyle=dotted,linecolor=red,linewidth=0.03cm]{-0.45}{0.5}{x x add 1 add x 1 add div }
\psplot[plotstyle=line,linestyle=dotted,linecolor=red,linewidth=0.03cm]{1}{1.45}{x x add 1 add x 1 add div }
\psplot[plotstyle=curve, linewidth=0.035cm]{0}{0.5}{3 x sub x sub 2 x sub div }
\psplot[plotstyle=line,linestyle=dotted,linecolor=red,linewidth=0.03cm]{-0.45}{0}{3 x sub x sub 2 x sub div }
\psplot[plotstyle=line,linestyle=dotted,linecolor=red,linewidth=0.03cm]{0.5}{1.45}{3 x sub x sub 2 x sub div }

\psdots[dotsize=0.06](0,1)
\psdots[dotsize=0.06](1,1)
\rput(-0.15,-0.15){$0$}
\rput(-0.21,1.3){\tiny $4/3$}
\rput(-0.21,1.455){\tiny $3/2$}
\rput(0.5,-0.1){\tiny $1/2$}
\rput(0,1.5){$\circ$}
\rput(1,1.5){$\circ$}
\psline[linewidth=0.02cm,linecolor=blue, linestyle=dashed,dash=0.08cm 0.08cm](0,1.31)(0.5,1.31)
\psline[linewidth=0.02cm,linecolor=blue, linestyle=dashed,dash=0.08cm 0.08cm](0.5,0)(0.5,1.31)
\rput(0,2.4){$\dim_H {\rm Sing}(w)$}
\rput(1.85,-0.15){$w_1$}
\end{pspicture} 
\caption{The Hausdorff dimension of ${\rm Sing}(w)$.}
\end{figure}
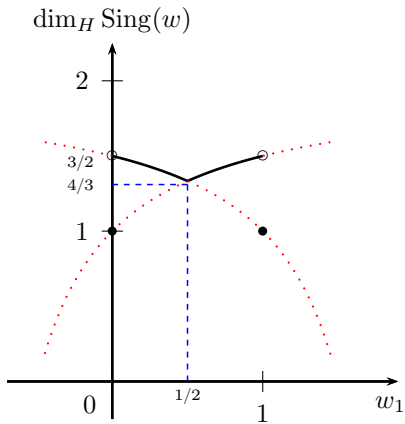

Here and hereafter we always assume that the weight vector $w$ satisfies the assumption of Theorem \ref{thm;main}.
It is  observed by Dani \cite{d85} that 
$w$-singular vectors correspond to   certain divergent trajectories in 
the space
$\mathcal L_3$ of unimodular lattices in $\R^3$ with respect to the
one-parameter semi-group
\begin{align}\label{eq;a t}
\mathcal A^+=\{a_t=\mathrm{diag}(e^{w_1t},e^{ w_2t},e^{ -t}): t\ge 0  \}.
\end{align} 
More precisely, 
$x\in \R^2$ is $w$-singular if and only if $\mathcal A^+ h(x)\Z^3$ is divergent where 
\begin{align}\label{eq;h x new}
 h(x)=\left(
 \begin{array}{ccc}
 1 & 0 & x_1 \\
 0 & 1 & x_2 \\
 0 & 0 & 1
 \end{array}
 \right).
\end{align}
Using this dynamical interpretation it is not hard to see that given $y\in \Q^2$, a vector  $x\in \R^2$ is $w$-singular if and only if 
$x+y$ is $w$-singular.
Therefore the conclusion of Theorem \ref{thm;main}
holds for any $U\cap \mathrm{Sing}(w)$ where $U$ is a nonempty   open subset of $\R^2$. 
Since $\Q^2\subset \mathrm{Sing}(w)$, the Minkowski dimension of $\mathrm{Sing}(w)$ is $2$ which is 
different from the Hausdorff dimension.

 The  lower bound  of Theorem \ref{thm;main} is proved
 by constructing a subset of $\mathrm{Sing}(w)$ with certain well-separated  self-affine structure using
 this dynamical interpretation. 
In fact, the denominator  $1+w_1 $ in the dimension formula  is the top Lyapunov exponent for the adjoint action of $\mathcal A^+$
 on the group 
 \begin{align}\label{eq;h x}
 \left \{ 
 h(x): x\in\R^2
 \right \}. 
 \end{align}
 This is reasonable since the top Lyapunov exponent corresponds to the shorter length of the rectangle in the 
 self-affine structure and the  shorter length is the length of the square  after chopping a rectangle into squares.
 Since our  construction has inductive nature it suffices to    look at the first 
 step to explain the ideas. We fix $t\gg 1$ and $\varepsilon<1$ with $\varepsilon \gg e^{-t} $ which means
 $\varepsilon^{-1} $ is negligible comparing to $e^{t}$. The lattice 
 $a_{t'} \Z^3$ moves to the cusp   in $\mathcal L_3$ as $t'$ goes from $0 $ to $t$, since the
 Euclidean  norm of 
 $a_{t'} \be_3$ where $\be_3=(0, 0, 1)\in \Z^3$ decays  exponentially. If we want  $x\in\R^2$
  satisfy $\|a_{t'} h(x) \be_3 \| \ll \varepsilon$ ($\|\cdot\|$ is the Euclidean norm) for $t'$  away from $0$, a reasonable condition  is $x\in U_0$
where  \[
 U_0=\{ (x_1, x_2)\in \R^2: |x_1|\le \varepsilon e^{-w_1t},  |x_2|\le\varepsilon e^{-w_2 t} \}.
 \]
 To play this game again we need $a_{t}h(x)\Z^3\in \mathcal L_3'$
 where 
 \begin{align}\label{eq;l prime}
 \mathcal L_3'=\{\Lambda\in \mathcal L_3: \Lambda\cap \R\be_3=r \Z\be_3 \mbox{ for  some } r 
 \mbox{ with } 1/2<r\le 1\}.
 \end{align}
 The cardinality of 
 $\{x\in U_0: a_{t} h(x)\Z^3\in \mathcal L_3' \}$ is up to some constants the cardinality of 
 $a_t\Z^3\cap M$ where $$M= \{(z_1, z_2, z_3)
 \in \R^3:|z_1|\le \varepsilon e^t |z_3| , |z_2|\le \varepsilon e^t|z_3|,  1/2<|z_3|\le 1 \}.$$
 It will follow from  lattice points counting  that this cardinality is approximately the area of 
 $M$ which is 
 $ \approx \varepsilon ^2 e^{2t}\approx e^{-t}\cdot  e^{3t}$.
 Here   $ e^{3t}$ is more or less the cardinality of the next subdivision of $U_0$ by rectangles of the size
 $2\varepsilon e^{-(2w_1+1)t}\times 2\varepsilon e^{-(2w_2+1)t}$,  so 
  it corresponds to the  full dimension $2$. 
In fact the   $-1$ in the numerator of $\frac{-1}{1+w_1}$ comes from the factor $e^{-t}$. 
To make the self-affine structure well-separated we need more conditions than just $a_th (x)\Z^3\in \mathcal  L_3'$ and the difficulty is to prove
that the cardinality of those $x$ is comparable to  $  \varepsilon ^2 e^{2t}$
using   geometry of numbers. 
In this part, 
different arguments are needed depending on whether $w_1>w_2$ or $w_1=w_2$. 
The lower bound is proved only in the  genuine weighted case $w_1>w_2$, since in the 
 unweighted case  the Hausdorff dimension is known.

Our proof of the   upper bound   follows 
 the similar  ideas of  \cite{c11}
and \cite{cc}. We use best approximation vectors with  weight $w$ to encode $\mathrm{Sing}(w)$
to get a self-affine covering of  the essential part of $\mathrm{Sing}(w)$. The main difference comparing to the unweighted case is that our
covering is self-affine instead of self-similar and this difference makes  the calculation more subtle.
In the unweighted case Einsiedler and Kadyrov \cite{ek12} have an estimate of the upper bound using entropy. 
The method in \cite{ek12} is further developed by Kadyrov, Kleinbock, Lindenstrauss and Margulis \cite{kklm} to estimate 
the upper bound  of the Hausdorff dimension of  general (unweighted) singular systems of   linear forms. The new input of this development is the use 
of the height function in  Eskin--Margulis--Mozes \cite{emm98} and its contracting property.
 Inspired by 
\cite{kklm}, it seems  that 
the number  $1$ in the numerator of $\frac{1}{1+w_1}$ might  also  be interpreted  as certain  average contracting rate of the  height function in \cite{emm98} with respect to $\mathcal A^+$ and the group
(\ref{eq;h x}).

Based on our interpretation of Theorem \ref{thm;main} and \cite[Corollary 1.2]{kklm} it seems likely   that the Hausdorff dimension of weighted 
 singular vectors in 
 $\R^d\ (d\ge 2)$ 
can be formulated  in a similar way. Namely, if we normalize the weights
so that the sum of positive weights  is equal to $1$, the Hausdorff dimension of weighted  singular 
vectors in $\R^d$ is 
\[
d-\frac{1}{\lambda_1}
\]
where $\lambda_1$ is the top Lyapunov exponent for the adjoint action of the corresponding one-parameter semi-group on the 
corresponding unipotent group.

Now  we turn to the  Hausdorff dimension of  vectors in $\R^2$ for which Dirichlet's theorem can be improved. For a positive real number  $\varepsilon<1$, we say $w$-weighted 
Dirichlet's theorem  is $\varepsilon $-improvable for $x\in \R^2$ if (\ref{eq;improve}) admits integer solutions
$(p, q)\in \Z^2\times \Z$
for $T$ sufficiently large.  Let $\mathrm{DI}(w, \varepsilon)$ be the set of vectors 
$x\in \R^2$ for which $w$-weighted Dirichlet's theorem is $\varepsilon$-improvable. 
It follows directly from the definition that 
\[
\mathrm{Sing}(w)=\bigcap_{0<\varepsilon <1} \mathrm{DI}(w, \varepsilon). 
\]
We remark here that in the unweighted case our set $\mathrm{DI}(w, \varepsilon)$ is $\mathrm{DI}_{\sqrt \varepsilon}(2)$ defined in 
\cite{c11} and \cite{cc}. 

Denote by $\dim_H$ the Hausdorff dimension. We have the following theorem.
\begin{thm}\label{thm;improve}
	Let $w=(w_1, w_2)$ where $w_1\ge w_2>0$ and $w_1+w_2=1$. There exists $C>0$ such that for all  $0< \varepsilon\le 2^{-5/w_2}$ one has  
	\begin{align}\label{eq;sun}
	2-\frac{1}{1+w_1}\le \dim_H \mathrm{DI}(w, \varepsilon)\le 2-\frac{1}{1+w_1}+C\sqrt \varepsilon. 
	\end{align}
\end{thm}
\begin{rem}
	The constant $C$ in Theorem \ref{thm;improve} is computable.
	In the unweighted case, the upper bound
	in (\ref{eq;sun}) is the same as \cite[Theorem 1.6]{c11}; while our method does not give good  lower bound 
	as in  \cite[Theorem 1.4]{cc}.
\end{rem}

Finally we  discuss  the divergent trajectories of $\mathcal A^+$ in  $\mathcal L_3$. 
We want to estimate the Hausdorff dimension of the set
\[
\mathcal D(\mathcal L_3, \mathcal A^+):=\{\Lambda\in \mathcal L_3: 
\mathcal A^+ \Lambda \mbox{ is divergent}\}. 
\]
Here the Hausdorff dimension is with respect to any Riemannian metric on the manifold 
$\mathcal L_3 \cong SL_3(\R)/SL_3(\Z)$. 
In the unweighted case the group (\ref{eq;h x}) is the unstable horospherical subgroup of $a_1$. Therefore as a corollary of the Hausdorff dimension 
of $\mathrm{Sing}(\frac{1}{2},\frac{1}{2})$ it is proved in \cite[Corollary 1.2]{c11} that the Hausdorff dimension of 
$\mathcal D(\mathcal L_3, \mathcal A^+)$
is $7\frac{1}{3}$. 
In authentic weighted case where $w_1>w_2$, the  unstable horospherical  subgroup 
of $a_1$  is the upper triangular unipotent group in $SL_3(\R)$
and  the group (\ref{eq;h x}) is  a proper subgroup of it.   
So we can not get the Hausdorff dimension of  divergent trajectories from Theorem \ref{thm;main}.
On the other hand, our method  for proving the lower bound of the Hausdorff dimension of $\mathrm{Sing}(w)$
can also be used to prove the following result. 
\begin{thm}\label{thm;slice}
	Suppose $w=(w_1, w_2) $ where $w_1+w_2=1$ and $w_1>w_2$. 
	For any  $\Lambda\in \mathcal L_3$ and any nonempty open subset $U$ of $\R^2$,   the Hausdorff dimension of 
	\[
	\{x\in U : \mathcal A^+ h(x)\Lambda
	\mbox{ is divergent}   \}
	\]
	 is at least $2-\frac{1}{1+w_1}$. 
	
\end{thm}
Theorem \ref{thm;slice} immediately implies the following corollary.
\begin{cor}\label{cor;div}
 	Let $w=(w_1, w_2)$ where $w_1> w_2>0$ and $w_1+w_2=1$. Then the Hausdorff dimension of 
 	$\mathcal D(\mathcal L_3, \mathcal A^+)$ is at least $8-\frac{1}{1+w_1}$. 
\end{cor}

We organize the paper as follows. In \S 2, we describe a fractal structure and develop some techniques for estimation of Hausdorff dimension.
\S 3 is devoted to counting lattice points in a convex subset of Euclidean space. In \S 4, we give the proof of the lower bound of the Hausdorff dimension of $\mathrm{Sing}(w)$ and the proof of Theorem \ref{thm;slice}. The proof of the upper bound of the Hausdorff dimension of $\mathrm{Sing}(w)$ and the proof of Theorem \ref{thm;improve} are given in the last section.

To make our presentation easier to follow, we give in Figures 2 and 3 the relations between the theorems for the lower and upper bounds  of  $\dim_H\mathrm{Sing}(w)$ 
respectively.

\begin{figure}
\centering
\begin{pspicture}(0,-3.2)(12,6.5)
\rput(1,6){Lemma \ref{1_step_bound}}
\rput(3.5,6){Lemma \ref{lem;count 1}}
\rput(5.5,6){Lemma \ref{lem;count 2}}
\psline{->}(7.3, 4.5)(6.5,4.5)
\rput(8.3,6){Lemma \ref{lem;teaching}}
\psline{->}(9.3, 4.5)(10.1,4.5)
\rput(11.1,6){Lemma \ref{lem;hyperplane}}
\psline{->}(8.3,4.8)(8.3, 5.7)
\psline{->}(11.1,5.7)(11.1, 4.8)
\rput(5.5,4.5){Lemma \ref{lem;primitive}}
\rput(8.3,4.5){Lemma \ref{cor;count 1}}
\psline{->}(8,4.8)(5.8, 5.7)
\psline{->}(8.6,5.7)(10.8, 4.8)
\rput(11.1,4.5){Lemma \ref{lem;technical}}
\psline{->}(10.8,4.2)(8.6, 3.4)
\psline{->}(1,5.7)(1, 4.8)
\psline{->}(5.5,5.7)(5.5, 4.8)
\psline{->}(3.5,5.7)(5.2, 4.8)
\rput(1,4.5){Lemma \ref{lem:measure_bound}}
\rput(3.1,4.53){Lemma \ref{claim:replace_elementary}}
\psline{->}(1,4.2)(1.8, 3.4)
\psline{->}(3,4.2)(2.2, 3.4)
\psline{->}(5.5,4.25)(5.5, 3.4)
\psline{->}(8.3,4.2)(8.3, 3.4)
\psline{->}(11.1,4.2)(11.1, 3.4)

\rput(2,3.1){Theorem \ref{thm:lower_bound}}
\psline{->}(2,2.8)(2, 1.9)
\psline{->}(8.3,2.8)(8.3, 1.9)
\psline{->}(5.5,2.8)(8, 1.9)
\psline{->}(10.8,2.8)(8.6, 1.9)
\rput(2,1.6){Corollary \ref{cor;real use}}
\psline{->}(2,1.3)(2, 0.4)
\rput(8.3,1.6){Lemma \ref{lem;card}}
\psline{->}(8.3,1.3)(8.3, 0.4)
\psline{->}(11.1,2.8)(11.1, 0.4)
\rput(5.5,3.1){Lemma \ref{lem;many vectors}}
\rput(8.3,3.1){Lemma \ref{lem;application 1}}
\rput(11.1,3.1){Lemma \ref{lem;application 2}}
\rput(2,0.1){Corollary \ref{cor;real real}}
\rput(8.3,0.1){Lemma \ref{lem;plenty}}
\rput(11.1,0.1){Lemma \ref{lem;separation}}

\psline{->}(2,-0.2)(4.5, -1.1)
\psline{->}(8.3,-0.2)(5, -1.1)
\psline{->}(11.1,-0.2)(5.8, -1.1)
\rput(5,-1.4){Proposition \ref{prop;transfer}}
\rput(9.7,-1.4){Lemma \ref{lem;contained}}
\psline{->}(5,-1.7)(6, -2.6)
\psline{->}(9.7,-1.7)(6.5, -2.6)
\rput(6.5,-2.9){Theorem \ref{thm;lower bound} (lower bound)}
\end{pspicture} 
\caption{The relations between theorems for the lower bound.}
\end{figure}
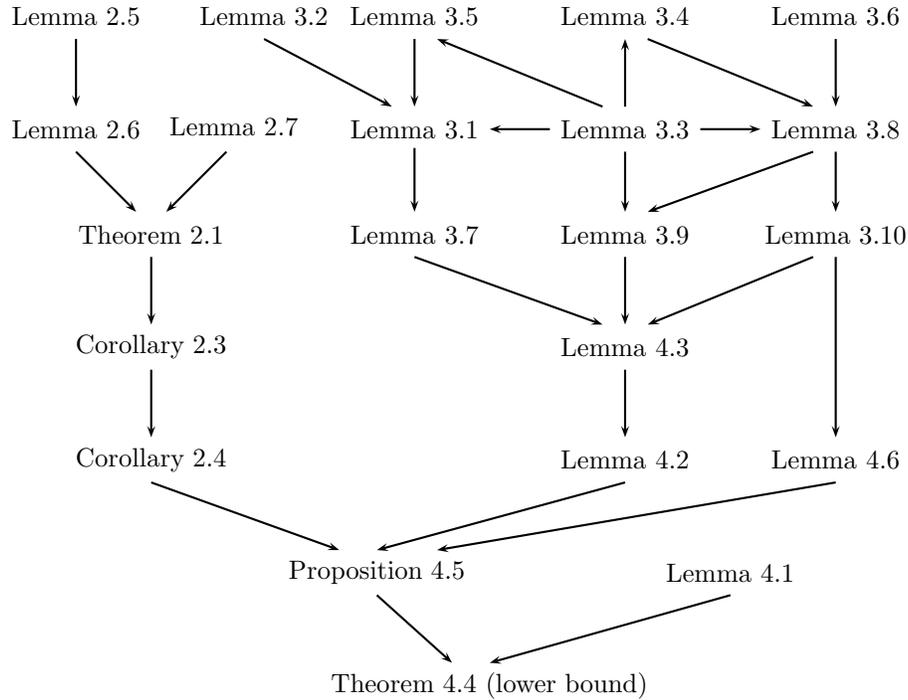

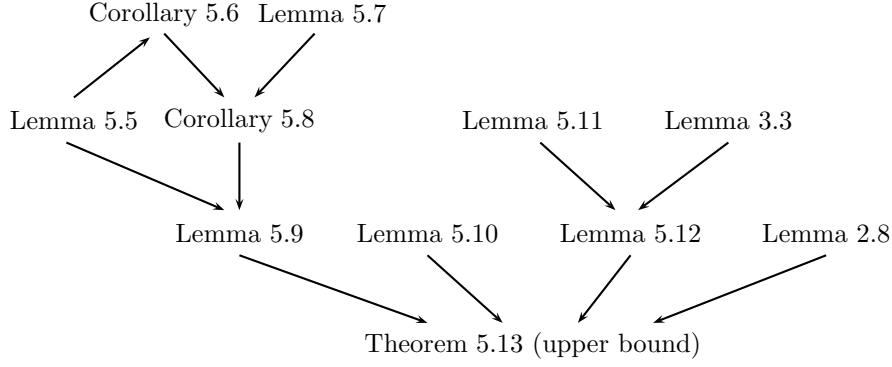
\begin{figure}
\centering
\begin{pspicture}(-1.5,-0.15)(12,5.5)
\rput(-0.2,3.1){Lemma \ref{lem;bound sequence}}
\psline{->}(-0.2,3.4)(0.81, 4.2)
\rput(1,4.5){Corollary \ref{cor;same size}}
\rput(3.1,4.53){Lemma \ref{lem;charcaterise}}
\psline{->}(1,4.25)(1.8, 3.4)
\psline{->}(3,4.25)(2.2, 3.4)
\rput(2,3.1){Corollary \ref{cor;singular}}
\psline{->}(2,2.8)(2, 1.9)
\psline{->}(-0.3,2.8)(1.8, 1.9)
\psline{->}(6,2.8)(7, 1.9)
\psline{->}(8.5,2.8)(7.3, 1.9)
\rput(2,1.6){Lemma \ref{lem;cover}}
\rput(4.5, 1.6){Lemma \ref{lemma-D}}
\rput(7.2,1.6){Lemma \ref{lemma-E}}
\rput(9.8,1.6){Lemma \ref{thm-general-upper}}
\rput(5.9,3.1){Lemma \ref{lem;equality}}
\rput(8.5,3.1){Lemma \ref{cor;count 1}}
\psline{->}(2,1.3)(4.5, 0.4)
\psline{->}(4.5,1.3)(5.5, 0.4)
\psline{->}(7.2,1.3)(6.5, 0.4)
\psline{->}(9.8,1.3)(7.5, 0.4)
\rput(5.9,0.1){Theorem \ref{thm;estimate} (upper bound)}
\end{pspicture} 
\caption{The relations between theorems for the upper bound.}
\end{figure}

\section{Fractal  structure  and Hausdorff dimension}
In this section we first review the description of a fractal structure using a rooted tree and
develop some techniques for estimating the lower bound of the Hausdorff dimension of a fractal set  from this structure. Then we prove 
an upper bound estimate theorem for a fractal set given by 
certain  relations.

\subsection{Fractal  structure}

Recall that a rooted tree is a connected graph $\mathcal T$ without cycles and with a distinguished vertex $\tau_0$, called the root of $\mathcal  T$. 
In this paper we identify $\mathcal T$ with the set of vertices of the tree $\mathcal T$. 
Any vertex  $\tau\in \mathcal  T$ is connected to 
$\tau_0$ by a unique path whose   length  is called the height of $\tau$.
The set of vertices of height $n$ is  denoted by $\mathcal T_n$.
Therefore $\mathcal T_0=\{ \tau_0\}$.
A vertex $\tau\in \mathcal T_{n}$ is  connected with a unique  $\tau_{n-1}\in \mathcal T_{n-1}$ and we  
say $\tau $ is a 
son of  $\tau_{n-1}$. 
The set 
of sons of $\tau\in \mathcal T$ is denoted by $\mathcal T(\tau)$.
A boundary point of $\mathcal T$ is 
a sequence of vertices    $\{\tau_n\}_{n\in \N}$ where 
$\tau_n$ is a son of $\tau_{n-1}$.   The boundary of $\mathcal T$ consists of all the 
boundary points and is denoted by $\partial \mathcal T$. 
For a vertex $\tau \in \mathcal T_n$ the set of ancestors of $\tau$
is
\[
\mathcal A(\tau):= \left\{ \tau_i: 0\le i\le n-1, \tau_{i+1}\in \mathcal T(\tau_{i}) \mbox{ where }\tau_n=\tau
\right\}.
\]

Let $Y$ be a Polish space, i.e.~ a separable completely metrizable topological space. 
A fractal structure on $Y$ is a pair $(\mathcal T, \beta)$
where $\mathcal T$ is a rooted tree and 
$\beta$ is a map from  $\mathcal T$ to the set of nonempty compact subsets of $Y$.
A fractal structure gives a set 
\[
\mathcal F(\mathcal T, \beta):=\{y\in Y: y\in \cap_{n=0}^\infty \beta(\tau_n) \mbox{ for some  } \{\tau_n\}\in \partial\mathcal T \},
\]
which is said to be the fractal with the structure $(\mathcal T, \beta)$. 
We remark that although each point of $\mathcal F(\mathcal T, \beta)$ should  correspond to  an infinite path, we do not assume each vertex of $\mathcal T$ has a son.
In particular if  $\mathcal T$  has only  finitely many vertices, then the fractal set $\mathcal F(\mathcal T, \beta)$
has to be  empty  according to our definition. 

We say that  $(\mathcal T, \beta)$ is a regular fractal structure if 
moreover the following properties hold:
\begin{itemize}
	\item  each vertex of $\mathcal T$ has a nonempty set of sons;
	\item if $\tau$ is a son of $\tau'$ then $\beta(\tau)\subset \beta(\tau')$;
	\item for any  $\{\tau_n\}\in \partial \mathcal T$ the diameter of 
	$\beta(\tau_n)$ goes to zero as $n$ tends to infinity.
\end{itemize}

We end up this section with several notations that will be used in the rest of the  paper. 
For a set $\mathcal S$ we use $\sharp \mathcal S$ to denote its cardinality. Let $A, B$ be two subsets of a metric space $(Y, d_Y)$, then 
\[
\mathrm{dist}\,(A, B):=\inf_{x\in A, y\in B} d_Y(x, y)\quad\mbox{and} \quad \mathrm{diam}\,(A):=\sup_{x, y\in A} d_Y(x, y). 
\]
We will only use the  above notation for the natural Euclidean metric on $\R^d$. For a real number $s$ we take
\[
\lceil s \rceil:=\inf\, \{n\in \Z: n\ge s\}\quad and \quad 
\lfloor s \rfloor :=\sup \, \{n\in \Z: n\le s\}. 
\]
For two nonnegative real numbers $s$ and $t$ the notation   $s \ll_{\mathcal S} t$ means that  there is a constant 
$C\ge 1$ possibly depending on elements of the set $\mathcal S$ such that $s\le C \,t$.
We call $C$ an implied constant for $s\ll_{\mathcal S} t$. 
  The notation 
$s\gg _{\mathcal S}t$ means $t\ll _{\mathcal S} s$ and the notation $s\asymp _{\mathcal S} t$
means both $s \ll_{\mathcal S} t$ and $s\gg _{\mathcal S}t$.

\subsection{Self-affine structure  and lower bound}\label{sec;lower}

 In this paper a rectangle means
a rectangle in $\R^2$ with sides  parallel to the axes. 
In particular,
a rectangle with  size $l_1\times l_2$ and center $x\in \R^2$ refers to the set  
\[
\{y\in \R^2: |y_1-x_1|\le l_1/2,\ |y_2-x_2|\le l_2/2 \}. 
\]
A self-affine structure on $\R^2$ is a fractal structure $(\mathcal T, \beta)$
on $\R^2$ such that 
for every $\tau \in \mathcal T$ the set $\beta(\tau)$ is a rectangle 
with  size  $W(\tau)\times L(\tau)$. 
A self-affine structure is said to be regular if 
the corresponding  fractal structure is regular.

The main result of this section is Theorem \ref{thm:lower_bound}.
What we are going to use in the lower bound calculation are Corollaries 
 \ref{cor;real use} and \ref{cor;real real} which are  simplified versions of the theorem.
We first prove 
these corollaries by assuming the theorem and then give the   proof of the theorem.

\begin{thm}
	\label{thm:lower_bound}
	Let $(\mathcal T, \beta)$ be a regular  self-affine  structure on $\mathbb R^2$.
	Suppose  there are sequences of positive  real numbers $\{W_n\},\{L_n\},  \{\rho_n \}, \{C_n \}$ indexed by $\N\cup \{0\}$ with 
	the following properties:
	\begin{enumerate}
		\item $W(\tau)=W_{n}$, $L(\tau)=L_{n}$ and $W_n\le L_n$ for all $n$
		and  $\tau \in \mathcal T_n$; 
		\item  $C_0=1$ and $\sharp \mathcal T(\tau)\geq C_{n}$ for all $n\in \N$ and  $\tau \in \mathcal T_{n-1}$;
		\item  $\rho_n \le 1$ for all $ n\in \N$. Moreover, for all
		$\tau_{n} \in \mathcal T_{n}$ and different $\tau, \kappa\in \mathcal T(\tau_{n})$ 
		\[
		\mathrm{dist}\left(\beta \left(\tau \right),\beta\left(\kappa \right)\right)\geq\rho_{n+1} W_{n}.
		\]
	\end{enumerate}
	Let 
	\begin{align*}
	P_{n}&=\prod_{k=0}^{n}C_{k},	\\	
	D_{n}&=\max\left\{ k\ge n\::\: L_{k}\ge W_{n}\right\},  \\
	s&=\sup\left\{ t >0\::\:\lim\limits _{n\to\infty}\frac{\log\left(P_n W_{n}^{ t }
	\rho_{n+1}^{ t  }
		\cdot \prod_{ i =n+1}^{D_{n}}\rho_{ i }C_i\right)}
	{\max\{D_{n}-n, 1\}}=\infty\right\} .
	\end{align*}
	If $s>1$, then $\dim_H \mathcal F(\mathcal T, \beta)\ge s$. 
	
\end{thm}
\begin{rem}

 If $D_n=n$  we interpret $\prod_{ i =n+1}^{D_{n}}\rho_i C_i=1$. Since $(\mathcal T, \beta)$ is regular, 
one has $W_n\to 0$ and hence $\rho_{n+1} W_n\to 0$. It follows that if $t<s$ then 
\[
\lim\limits _{n\to\infty}\frac{\log\left(P_{n}  W_{n}^{ t }
	\rho_{n+1}^{ t  }
	\cdot\prod_{ i =n+1}^{D_{n}}\rho_{ i }C_i\right)}
{\max\{D_{n}-n, 1\}}=\infty.
\]
The main difference between Theorem \ref{thm:lower_bound} and  lower bound theorems used in \cite{c11} and \cite{cc} is the factor $\prod_{ i =n+1}^{D_{n}}\rho_i C_i$ which is trivial  for the usual  self-similar  
structures (here self-similar refers to  $W_n=L_n$). 

\end{rem}

The formula of the  lower bound $s$ in Theorem   \ref{thm:lower_bound} 	 is much simpler 
in many interesting self-similar fractal structures
where $L_n=W_n$ and  
$D_n=n$ for all $n\in \N$. The following two corollaries  on   refinement of the lower bound formula
are only  interesting 
in 
authentic self-affine cases  where the following assumption (\rmnum{4}) of Corollary \ref{cor;real use} is needed.

\begin{cor}\label{cor;real use}
	Let the notation be as in  Theorem \ref{thm:lower_bound}. We moreover assume that there
	exists $k\in \N$  such that for all sufficiently large  $n\in \N$   the following conditions hold:
	(\rmnum{1})
	$
	D_n \le k n
	$; 
	(\rmnum{2})
	$e^{n/k}\le C_n\le e^{kn}$;
	(\rmnum{3})
	$\rho_n\ge e^{-nk}$; 
	(\rmnum{4})
	$\rho_nC_nL_n/L_{n-1}\ge n^{-k}$.
	Suppose that
	\begin{align}\notag
	s=\sup\left\{  t >0:\lim\limits _{n\to\infty}\frac{L_{n}}{W_{n}}P_n\cdot W_{n}^{ t }=\infty\right\}  =\liminf_{n\to\infty}\frac{\log\left(L_{n}P_n\right)}{-\log W_{n}}+1
	\end{align}
	is strictly bigger than $1$, then
	$\dim_H \mathcal F(\mathcal T, \beta)\ge s$.	
\end{cor}

The conclusion of this corollary implies that  the Hausdorff dimension of $\mathcal F(\mathcal T, \beta)$
is  equal to the lower Minkowski dimension.
Before the proof we explain the  additional assumptions of the corollary.  The assumptions
(\rmnum{2}) and (\rmnum{3}) are standard   and they are  satisfied by the usual  fractal structures of  Cantor sets. The assumption (\rmnum{1}) is 
a regularity condition which  means that  $D_n$ grows at most linearly in $n$. If we fix $\tau\in \mathcal T_{n-1}$ and enlarge all the $\beta(\kappa)$ for  $\kappa\in \mathcal T(\tau)$ to rectangles with the same centers and size 
$(W_{n}+\frac{1}{4}W_{n-1}\rho_n)\times (L_n+\frac{1}{4}W_{n-1}\rho_n)$, then they are mutually disjoint. 
Therefore, by assumption (2) in Theorem \ref{thm:lower_bound}  
\[
C_n\le  \sharp \mathcal T(\tau)\le \frac{W_{n-1 }L_{n-1} }{(W_{n}+\frac{1}{4}W_{n-1}\rho_n) (L_n+\frac{1}{4}W_{n-1}\rho_n)}.
\]
The assumption (\rmnum{4}) is satisfied if  $L_n\gg \rho_n W_{n-1}\gg  W_n$ and the above inequality is 
almost an equality up to sub-exponential factors. So (\rmnum{4}) means that the separation of  $\beta(\kappa)$ $(\kappa \in \mathcal T(\tau))$ is 
almost  optimal (see the  remark after Lemma \ref{lem;separation} for more explanations). 
\begin{proof}[Proof of Corollary \ref{cor;real use}]
	It follows from the 	assumptions (\rmnum{1})-(\rmnum{4}) that 
	given $t>0$ and  $\varepsilon>0$  one has 
	\begin{align}\label{eq;too cold 5}
	\min \, \left \{(\rho_{D_n+1}C_{D_n+1})^{-1},\  
	\rho_{n+1}^t  
	\prod_{i=n+1}^{D_n}\frac{\rho_iC_i L_i}{L_{i-1}}\right \}\ge P_n^{-\varepsilon}
	\end{align}
	for all  $n$  sufficiently large (depending on $t$ and $\varepsilon$). 
	We fix a real number  $t$   with $1< t< s$. Let $\varepsilon>0$ be  sufficiently  small so that 
	$1<\frac{t}{1-3\varepsilon}<s$.  
	It follows from the definition of $s$ that 
	\begin{align}\label{eq;very cold}
	\frac{L_{n}}{W_{n}}P_n\cdot W_{n}^{ \frac{t}{1-3\varepsilon} } \ge 1
	\end{align}
	for $n$ sufficiently large. 
	Let $n_0\in \N$ such that  (\ref{eq;too cold 5})  and (\ref{eq;very cold}) hold  for $n\ge n_0$. 
	Then for all $n\ge n_0$ we have 
	\begin{align*}
	P_n W_n^t 
	\rho_{n+1}^t 
	\prod_{i=n+1}^{D_n}\rho_i C_i& \ge 
	P_n^{1-\varepsilon} W_n^t 
	\rho_{n+1}^t 
	\prod_{i=n+1}^{D_n+1}\rho_i C_i&& \mbox{by (\ref{eq;too cold 5})}\\
	&=  P_n^{1-\varepsilon} W_n^t\cdot
	\rho_{n+1}^t
	\frac{L_{n}}{L_{D_n+1}} \cdot
	\prod_{i=n+1}^{D_n+1}\frac{\rho_iC_i L_i}{L_{i-1}}   \\
	&\ge P_n^{1-\varepsilon} W_n^t \cdot \frac{L_n}{W_n} \cdot  
	\rho_{n+1}^t 
	 \prod_{i=n+1}^{D_n+1} \frac{\rho_iC_i L_i}{L_{i-1}} \\
	&\ge P_n^{1-\varepsilon} W_n^t \cdot \frac{L_n}{W_n}\cdot  P_n^{-\varepsilon} && \mbox{by (\ref{eq;too cold 5})}\\
	& \ge P_n^{1-\varepsilon} W_n^t\cdot  \left (\frac{L_n}{W_n}\right)^{1-3\varepsilon}\cdot  P_n^{-2\varepsilon}\cdot P_n^{\varepsilon} \\
	&=\left [ P_n W_n^{\frac{t}{1-3\varepsilon}}\frac{L_n}{W_n}\right]^{1-3\varepsilon} P_n^{\varepsilon}\\
	&\ge P_n^{\varepsilon}  && \mbox{by (\ref{eq;very cold})}.    
	\end{align*}
	Therefore, the assumptions  (\rmnum1) (\rmnum2)  and Theorem \ref{thm:lower_bound}    imply  $$\dim_H\mathcal F(\mathcal T, \beta)\ge t.$$ The conclusion follows by  considering an arbitrary
	real number $t$ with   $1< t<s$.

\end{proof}

The assumptions (\rmnum{2})-(\rmnum{4}) in Corollary \ref{cor;real use}  are local, that is, they only depend on 
the data from height  $n-1$ to  $n$. The Hausdorff dimension of $\mathcal F(\mathcal T, \beta)$ 
can also be estimated via local data under an additional assumption. We state this observation as the following corollary. 
\begin{cor}\label{cor;real real}
	Let the notation be as in  Theorem \ref{thm:lower_bound}. We moreover assume that there
	exist $k,  n_0\in \N$   such that for all   $n\ge  n_0$ (\rmnum{2})-(\rmnum{4}) in Corollary \ref{cor;real use} hold,  $L_{kn}/L_{kn-1}\le  W_n/W_{n-1}$ and 
	$L_{kn_0-1}< W_{n_0-1}$. If
	\[
	\lim_{n\to \infty}  \frac{\log ({L_n} C_n/{L_{n-1}} )}{-\log (W_n/W_{n-1})}
	\]
	exists and is  equal to $r>0$, then $\dim_H\mathcal F(\mathcal T, \beta)\ge 1+r$.
\end{cor}
\begin{proof}
	Since the sequence $\{ L_n\}$ is monotonically decreasing, for  $n\ge n_0$ we have 
	\begin{align*}
	L_{kn}& =L_{kn_0-1} \prod_{i=kn_0}^{kn} \frac{L_{i}}{L_{i-1}} 
	\le L_{kn_0-1}	\prod_{i=n_0}^n \frac{L_{ki}}{L_{ki-1}} \\
	& \le  L_{kn_0-1}\prod_{i=n_0}^n\frac{W_{i}}{W_{i-1}} 
	= W_n \frac{L_{kn_0-1}}{W_{n_0-1}}< W_n.
	\end{align*}
	So the  assumption (\rmnum{1}) of Corollary \ref{cor;real use} holds and we
	conclude that  the 
	Hausdorff dimension of $\mathcal F(\mathcal T, \beta)$ is bounded from below by $1+r$. 
	
\end{proof}

In the rest of this section we keep   the notation and assumptions of 
Theorem \ref{thm:lower_bound} and give a proof of it. 
We first develop  some tools for the proof of Theorem \ref{thm:lower_bound}. 
According to the assumptions there is a one-to-one correspondence between
$\mathcal F(\mathcal T, \beta) $ and  $\partial \mathcal T$. For each 
$x\in \mathcal F(\mathcal T, \beta) $  we let $\{\tau_n(x)\}_{n\in \N}\in \partial \mathcal T$  such that 
$\bigcap_{n\in \N} \beta(\tau_n(x) )=\{x\}$. 
We  take
$\tau_0(x)$ to be the root of $\mathcal T$ for all $x$.
Let $\mu$ be the measure on $
\mathcal F(\mathcal T, \beta)$ with the property that for all $y\in \mathcal F(\mathcal T, \beta)$ and $n\in \N$
\[
\frac{ \mu(\{x\in  \mathcal F(\mathcal T, \beta): \tau_{n}(x)=\tau_{n}(y)  \})}
{ \mu(\{x\in  \mathcal F(\mathcal T, \beta): \tau_{n-1}(x)=\tau_{n-1}(y)  \})}= \frac{1}{\sharp \mathcal T(\tau_{n-1}(y))}\le \frac{1	}{C_n}.
\]
For any $\tau\in \mathcal T$, we define  {\em elementary squares} of $\beta(\tau)$
to be the closed squares contained in $\beta(\tau)$  with side-length
$W(\tau )$. In the following two lemmas we estimate the measure of an elementary square.
\begin{lem}\label{1_step_bound} 
	Suppose $n\in \N\cup \{0 \}$ and  $D_n>n$.
	Let $\kappa\in \mathcal T_n$ and $\tau \in \mathcal T_{i-1}$ where 
	$ n+1\le i\le D_n$. Then  for     any elementary square 
	$S$ of $\beta(\kappa)$  one 
	has  
	\[
	\sharp\{\tau'\in\mathcal T(\tau )\::\: \beta\left(\tau'\right)\cap S\neq\emptyset\}\le 72\rho_{ i }^{-1}.
	\]
\end{lem}
\begin{proof}
	Let 	$R_0=\beta\left(\tau \right)\cap S$ { and }
	\[
	\mathcal S=\left\{ \beta\left(\tau'\right)\cap S\::\: \tau'\in\mathcal T\left(\tau \right),\: \beta\left(\tau'\right)\cap S\neq\emptyset\right\} .
	\]
	Without loss of generality we assume $R_0$ is nonempty. Then $R_0$ is a rectangle with  size 
	$l_1\times l_2$ where
	$l_1= \min \{ W_n,  W_{i-1} \}=W_{i-1}$ and   $l_2\le \min \{ W_n,  L_{i-1} \}=W_n$. 
	Each  $R\in \mathcal S$ has  size $W_{ i }\times l(R)$ where $l(R)\le W_n$ and 
	the distance of  
	two different elements of $\mathcal S$ is  
	at least $W_{ i -1}\rho_{ i }$. 
	For every $R\in \mathcal S$ let $R'$ be the  rectangle with the same center and size  $  
	(W_{i}+\frac{\rho_i}{4} W_{i-1})\times W_n$.
	Similarly, let $R_0'$ be the rectangle with the same center as $R_0$ and size  $3 W_{i-1}\times 3W_n$. 
	Each point of $R'_0$ is covered by at most two  rectangles of $\{R': R\in \mathcal S \}$
	(here we use $L_i\ge  W_n$) and every $R'$
	is contained in $R_0'$.
	Therefore 
	\[
	\ (W_{i}+\frac{\rho_i}{4} W_{i-1})W_n \cdot 	\sharp \mathcal S\le  18 \ W_n W_{i-1},
	\]
	which implies 
	\begin{eqnarray*}
		\sharp\mathcal S
		\le 72 \frac{W_{n}}{W_n
		}\cdot\frac{W_{ i-1 }}{W_{ i -1}\rho_{i}}
		= 72\rho_{ i }^{-1}.
	\end{eqnarray*}
\end{proof}
\begin{lem}\label{lem:measure_bound}
	Let $n\in \N \cup \{0\}$ and   $\kappa \in \mathcal T_{n}$. Then for 
	any elementary square $S$ of $\beta(\kappa )$ one has 
	\begin{align}\label{eq;measure 1}
		\mu\left(S\right)\leq 72^{D_{n}-n}P_n^{-1}\prod_{ i =n+1}^{D_{n}}\rho_{ i }^{-1}C_i^{-1}.
	\end{align}
\end{lem}
\begin{proof}
	It is easy to see  that (\ref{eq;measure 1}) holds if 
	$D_n=n$. In the rest of the proof we assume   $D_n>n$. 
	Applying  Lemma \ref{1_step_bound} for $i=n+1,n+2,   \cdots , D_n$, we get
	\begin{align}\label{eq;starbucks}
		\sharp\left\{ \tau\in \mathcal T_{D_n}\::\: \beta\left(\tau\right)\cap S\neq\emptyset\right\} \le 72^{D_{n}-n}\prod_{ i =n+1}^{D_{n}}\rho_{ i }^{-1}.
	\end{align}
	We can cover $S\cap \mathcal F(\mathcal T, \beta)$ with the rectangles $$\left\{ \beta\left(\tau \right)\::\: \tau \in 
	\mathcal T_{D_{n}},\: \beta\left(\tau \right)\cap S\neq\emptyset\right\}. $$
	Therefore 
	\begin{align*}
		\mu\left(S\right) & \leq  \sum\limits _{\substack{\tau \in \mathcal T_{D_{n}} \\
				\beta\left(\tau\right)\cap S\neq\emptyset
			}
		}\mu\left(\beta\left(\tau \right)\right) \\
		&\le\mu\left(\beta\left(\kappa \right)\right)\prod_{ i =n+1}^{D_{n}}\frac{1}{C_{ i }}\cdot \sharp\left\{ \tau\in \mathcal T_{D_{n}}\::\: \beta\left(\tau\right)\cap S\neq\emptyset\right\} .
	\end{align*}
	The above inequality, (\ref{eq;starbucks}) and the fact  $\mu\left(\beta\left(\kappa \right)\right)\le P_n^{-1}$ imply 	(\ref{eq;measure 1}).
\end{proof}

Let $U$ be  an open subset of $\R^2$ with $U\cap \mathcal F(\mathcal T, \beta)\neq \emptyset$.  If  $U\cap \mathcal F(\mathcal T, \beta)$ contains at least two points   we let 
$n\left(U\right)$ be the largest  index $n\ge 0$ such that $U\cap \mathcal F(\mathcal T,\beta )\subset \beta\left(\tau\right)$
for some $\tau \in \mathcal T_n$. If $U\cap \mathcal F(\mathcal T, \beta)$  contains a single point we let  
$n(U)$ be the largest index $n\ge 0$ such that   $\mbox{diam}\left(U\right)\ge
\rho_{n+1}
W_{n}$. 
Then in $\mathcal T_{n(U)}$, there is a unique element denoted by $\kappa(U) $, such that $U\cap \mathcal F(\mathcal T, \beta)\subset \beta(\kappa(U))$. 

\begin{lem}
	\label{claim:replace_elementary}
	Let   $U$ be an open subset   of $\R^2$ with $U\cap \mathcal F(\mathcal T, \beta)\neq \emptyset$. 
	Let   $n=n\left(U\right)$ and  $\kappa=\kappa(U)$. 
	There is a family $\mathcal{S}$ of elementary squares of $\beta(\kappa)$ 
	such that 
	\begin{align}\label{eq;square}
		\bigcup_{S\in\mathcal{S}}S\supset U\cap \mathcal F(\mathcal T, \beta)\quad
		\text{ and }
		\quad	  W_n^t\cdot \sharp{ \mathcal{S}}\le 2
		\rho_{n+1}^{- t  }
		\mathrm{diam}\left(U\right)^{ t  }
	\end{align}
	for all $t\ge 1$.
\end{lem}
\begin{proof}
	We claim that $\mbox{diam}\left(U\right)\ge\rho_{n+1}W_{n}$. In fact, if $U\cap \mathcal F(\mathcal T, \beta)$ contains a single 
	point then the claim follows directly from the definition  of $n(U)$. 
	Otherwise,  there are at least two elements
	$\tau\neq \tau'\in\mathcal T\left(\kappa\right)$ for which $\beta\left(\tau\right),\beta\left(\tau '\right)$
	intersect $U$, and hence $\mbox{diam}\left(U\right)\ge\rho_{n+1}W_{n}$ by the assumption (3) of 
	Theorem \ref{thm:lower_bound}.
	
	If $\mbox{diam}\left(U\right)\leq W_{n}$, then there is an elementary  square $S$ of $\beta(\kappa)$ such that 
	$U\cap \mathcal F(\mathcal T, \beta)\subset S$ and we may take $\mathcal{S}=\left\{ S\right\} $.
	Then 
	\[
	 W_n^t\cdot  \sharp \mathcal S\le
	\rho_{n+1}^{- t  }
	\mbox{diam}\left(U\right)^t.
	\]
	If $\mbox{diam}\left(U\right)>W_{n}$, then there is a cover of $U\cap \mathcal F(\mathcal T, \beta)$ by
	$\left\lceil \frac{\mbox{diam}\left(U\right)}{W_{n}}\right\rceil $
	elementary squares. In this case,
	\begin{eqnarray*}
	 W_n^t \cdot \sharp \mathcal S& = & \left\lceil \frac{\mbox{diam}\left(U\right)}{W_{n}}\right\rceil 
			W_{n}^{ t  }\le2\left(\frac{\mbox{diam}\left(U\right)}{W_{n}}\right)W_{n}^{t }\\
		& \leq & 2\left(\frac{\mbox{diam}\left(U\right)}{W_{n}}\right)^{ t  }W_{n}^{ t  }\le 2
		\rho_{n+1}^{-t}
		\mbox{diam}\left(U\right)^t,
	\end{eqnarray*}
	where in the last two inequalities we use the assumptions  $t\ge 1$
	 and  $\rho_{n+1}\le 1$ in  	  (3) of Theorem	\ref{thm:lower_bound}. 
	
\end{proof}


\begin{proof}[Proof of Theorem \ref{thm:lower_bound}]
	Let  $t$ be any real number such that   $ 1\le t<s$.
	By the definition of $s$
	there exists $n_0=n_0(t)$ such that for all $n\ge n_0$ one has
	\begin{align}\label{eq;lowerbound}
		P_nW_n^t
		\rho_{n+1}^t
		 \cdot \prod_{i=n+1}^{D_n}\rho_iC_i\ge 72^{\max\{D_n-n, 1\} }\ge 72^{D_n-n}.
	\end{align}
	Suppose $\mathcal U $ is an open cover of $\mathcal F(\mathcal T, \beta)$.
	We  assume that the diameters of all
	elements in $\mathcal U $ are small enough so that $n\left(U\right)>n_{0}$ for all $U\in \mathcal U$.
	Since $\mathcal F(\mathcal T, \beta)$ is compact, there is a finite subcover $\mathcal{U}_{0}$
	such that each element of $\mathcal{U}_{0}$ has a nonempty intersection with $\mathcal F(\mathcal T, \beta)$.

	Using Lemma  \ref{claim:replace_elementary}, for every $U\in \mathcal U_0$ there is a set 
	$\mathcal S_U$ of elementary squares of  $\beta(\kappa (U))$ such that (\ref{eq;square}) holds 
	for $\mathcal S=\mathcal S_U$. Let $\mathcal Q=\bigcup_{U\in \mathcal U_0}\mathcal S_U$ and $n(S)=n(U)$ for 
	$S\in \mathcal S_U$.  We note here that although it is possible that the same $S$ belongs to different
	$\mathcal S_U$, the number $n(S)$ is well-defined. 
	Then $\mathcal Q$ is a covering of $\mathcal F(\mathcal T, \beta)$ and  
	\begin{align*}
		\sum_{U\in\mathcal{U}}\mbox{diam}\left(U\right)^{ t  }&\ge \frac{1}{2}\sum_{S\in\mathcal{Q}}
		\rho_{n\left(S\right)+1}^{ t  }
		W_{n(S)}^t  && \mbox{by (\ref{eq;square})}\\
		&\ge \frac{1}{2}\sum_{S\in \mathcal Q} 72^{D_{n(S)}-n(S)}P_{n(S)}^{-1}\prod_{ i =n(S)+1}^{D_{n(S)}}\rho_{ i }^{-1}C_i^{-1}  && \mbox{by (\ref{eq;lowerbound})}\\
		& \ge \frac{1}{2}\sum_{S\in \mathcal Q} \mu(S) && \mbox{by (\ref{eq;measure 1})}\\
		& \ge \frac{1}{2}. 
	\end{align*}
	Therefore, $\dim_H(\mathcal F(\mathcal T, \beta))\ge  t $. By considering an arbitrary $t$ with $1\le t< s$ we have 
	$\dim_H(\mathcal F(\mathcal T, \beta))\ge  s $. 
	
\end{proof}

\subsection{Fractal  relation and upper bound}

Let $Q$ be a countable set. 
We call 
a subset   $\sigma$ of $Q^2= Q\times Q$   a relation on $Q$.
For each $\tau \in Q$ we let $\sigma(\tau)=\{\kappa \in Q: (\tau,\kappa )\in \sigma \}$. 
We write $\kappa \prec \tau$ if either $\kappa =\tau$ or 
 there exist  
$\tau_1, \ldots, \tau_n\in  Q$ such that 
$\tau_1=\tau$, 
$\tau_n=\kappa $
and  $(\tau_i, \tau_{i+1})\in \sigma $ for all $1\le i< n$.
The boundary of $\sigma $ is defined as 
\[
\partial \sigma= \{ \{\tau_i \}_{i\in \N}: (\tau_i, \tau_{i+1})\in \sigma    \}.
\]

A triple   $(Q, \sigma, \beta)$ is said to be a  fractal relation on a Polish space  $Y$ if 
$\beta $  is  a map from    $Q$ to nonempty  compact  subsets of $Y$ and 
$\sigma $ is 
a relation on   $Q$.
Moreover, we say  $(Q, \sigma, \beta)$ is  admissible if 
 $\mathrm{diam}\,  \beta(\kappa)
	<\mathrm{diam}\,  \beta(\tau)$ for any $(\tau,\kappa)\in \sigma$ and 
$\mathrm{diam}\,  \beta(\tau_i)\to 0$ as $i\to \infty$
	for   any sequence $\{ \tau_i\}_{i\in\N}\in \partial \sigma$. 
A fractal   relation $(Q, \sigma, \beta )$
gives a fractal  set 
\[
\mathcal F(Q, \sigma, \beta ):=\{ y\in Y: 
\{y\}= { \cap}_{i\in \N } \beta(\tau_i) \mbox{ for some } {\{\tau_i\}_{i\in \N}\in \partial \sigma} \}. 
\]

The following lemma is a self-affine version of \cite[Theorem 3.1]{c11}. 
\begin{lem}\label{thm-general-upper}
	Let $(Q, \sigma, \beta)$ be an admissible  fractal relation on $\R^2$ such that
	for every $\tau \in Q$ the compact set  
	$\beta(\tau)$ is a rectangle with size $W(\tau)\times L(\tau)$
	where $W(\tau)\le L(\tau)$. 
	Suppose  $s$ is  a positive real number with     
	\begin{align}\label{eq;liao assume}
		\sum_{\kappa \in \sigma(\tau )} L(\kappa )\cdot W(\kappa )^{s-1} \leq L(\tau)\cdot W(\tau)^{s-1}
	\end{align}
	for all $\tau\in Q$,
	then $\dim_H \mathcal F(Q, \sigma, \beta) \leq s$.
\end{lem}
\begin{proof}
For $\tau_0\in Q$ let 	
	$\mathcal F(\tau_0)=  \{   
	{ \cap}_{i\in \N } \beta(\tau_i) : {\{\tau_i\}_{i\in \N}\in \partial \sigma, \tau_1=\tau_0}\}$.
	Since $\mathcal F(Q, \sigma, \beta)$ is a countable union of $\mathcal F(\tau_0) \ (\tau_0\in Q)$, it suffices to 
	 show that  $\dim_H \mathcal F(\tau_0)\leq s$
	for all $\tau_0\in Q$.

	We fix $\tau_0$ and assume that $\mathcal F(\tau_0)\neq \emptyset$.  
	For $0<\varepsilon< {\rm diam}\, \beta(\tau_0) $  we		
	will find an $\varepsilon $-covering $\mathcal{U}$ of $\mathcal F(\tau_0)$ such that $\sum_{U\in \mathcal U} \mathrm{diam}\,( U) ^s $ is bounded from above by a finite number independent of 
	$\varepsilon$.  This will imply $\dim_H \mathcal F(\tau_0)\leq s$.
	
	Let 
	\[
	\mathcal S=\{ \tau\in Q:  {\rm diam}\, \beta(\tau) \le  \varepsilon  \mbox{ and }  {\rm diam}\, \beta(\kappa)> \varepsilon  \mbox{ for some }\kappa \prec \tau_0 \mbox{ with }
	(\kappa, \tau)\in \sigma \}.
	\]
	Each rectangle $\beta (\tau )$ is covered  by $r(\tau ):=\lceil L(\tau)/W(\tau)\rceil$ elementary  squares  with sides 
	$W(\tau )$.  Denote these pieces by $\{S_{\tau, i} : 1\leq i \leq r(\tau )\}$. 
	Since $(Q, \sigma,\beta)$ is  admissible and $\varepsilon< {\rm diam}\, \beta(\tau_0)$, the set  $\mathcal F(\tau_0)$ is covered by 
	\[
	\bigcup_{\tau  \in \mathcal S} \beta(\tau) = \bigcup_{\tau  \in \mathcal S} \bigcup_{1\leq i \leq r(\tau )} S_{\tau , i},
	\] 
	which is an $\varepsilon$-covering. 
	So it suffices to show 
	\begin{equation}\label{upper-cov}
		\sum_{\tau  \in\mathcal  S} \sum_{1\leq i \leq r(\tau )} \mathrm{diam}\,(S_{\tau , i})^s 
		\le  2^{s+1}L(\tau_0)W(\tau_0)^{s-1}.
	\end{equation}
	
	We 
	first note that for each $\tau \in \mathcal{T}$
	\begin{align*}
		\sum_{1\leq i \leq r(\tau )} \mathrm{diam}\, (S_{\tau , i})^s  = r(\tau ) \cdot (\sqrt{2} \cdot W(\tau))^s 
		\leq 2^{s+1} \cdot {L(\tau ) \over W(\tau )}\cdot W(\tau)^s.
	\end{align*}
	Hence,
	\begin{equation}\label{upper-cov2}
		\sum_{\tau \in\mathcal  S} \sum_{1\leq i \leq r(\tau )}  \mathrm{diam}\, (S_{\tau , i})^s\leq  2^{s+1} \sum_{\tau  \in\mathcal  S}  L(\tau )\cdot W(\tau )^{s-1}.
	\end{equation}

	We claim that 
	\begin{align*}
		\sum_{\tau  \in\mathcal  S}  L(\tau )\cdot W(\tau )^{s-1}\le L(\tau_0 )\cdot W(\tau_0 )^{s-1}. 
	\end{align*}
	Suppose the contrary, then there exists a finite set 
	$
	\mathcal S'=\{\tau_i: 1\le i\le k\}\subset \mathcal S 
	$
	such that 
	\begin{align}\label{eq;liao 2}
		\sum_{\tau  \in \mathcal S'}  L(\tau )\cdot W(\tau )^{s-1}> L(\tau_0 )\cdot W(\tau_0 )^{s-1}. 
	\end{align}
	According to the definition of $\mathcal S$, for each $\tau_{i}\in \mathcal S'$, there exists a finite  sequence $\tau_{i, j} \ (0\le j\le n_i)$ such that 
	$(\tau_{i, j-1}, \tau_{i, j})\in \sigma  \ (1\le j\le n_i)$, $\tau_{i, 0}=\tau_0$, $\tau_{i, n_i}=\tau_i$ and $\mathrm{diam}\, (\beta (\tau_{i, n_i-1}))> \varepsilon. $

	For each   $ 0\le j\le n:=\max_{1\le i\le k} n_i$ let    $$\mathcal 	S_j=
	\{ \tau_{i, j}: 1\le i\le k,  n_i \ge j\}\dot{\cup} \{    
	\tau_i : 1\le i\le k, n_i<  j \} , $$
	where $\dot{\cup}$ denotes the  disjoint union. 
Note that for  $1\le j\le n $ 
	\begin{align*}
	\mathcal S_{j-1}= \{\tau_{i, j-1}: 1\le i\le k, n_i\ge  j  \}\dot{\cup} \{    
	\tau_i : 1\le i\le k, n_i<  j \}.
	\end{align*}
	The union of $\sigma(\tau)$ for $\tau$ runs over $\{\tau_{i, j-1}: 1\le i\le k, n_i\ge  j  \}$ contains $\{ \tau_{i, j}: 1\le i\le k,  n_i \ge j\}$. 
	Therefore, 
	 (\ref{eq;liao assume}) implies 
	\begin{align}\label{eq;2018}
	\sum_{\tau  \in \mathcal S_j}  L(\tau )\cdot W(\tau )^{s-1}\le \sum_{\tau  \in
		\mathcal  S_{j-1}}  L(\tau )\cdot W(\tau )^{s-1} \quad (1\le j\le n).
	\end{align}
	Observing $\mathcal S_0=\{\tau_0 \}$ and $\mathcal S_n=\mathcal S'$, we deduce from (\ref{eq;2018}) a contradiction to (\ref{eq;liao 2}). The claim then follows.
		The claim  together with (\ref{upper-cov2}) imply (\ref{upper-cov}), which completes  the proof. 

\end{proof}

\section{Counting lattice points in convex sets}\label{sec;count}
The aim of this section is to develop some tools for counting lattice points in a convex subset of  the Euclidean space
$\mathcal E_d=\R^d$. 
 Although we only need these results  in the case where $d\le 3$, 
 we give some results  in general Euclidean space in \S \ref{sec;count rd}.  The reason for this is that  the proofs are the same and they might be useful 
 in other contexts. 
 Results in \S \ref{sec;count r3}
will only be used in
  our   estimation 
of the lower bound.
Since this section contains technical results that we will use later, 
 the reader can skip this
section in the first reading and come back when needed.

\subsection{Lattice points counting in $\R^d$}\label{sec;count rd}

Let  $K$ be a bounded  centrally symmetric convex  subset of $\R^d$ with nonempty interior
and let $\Lambda$ be a lattice of $\R^d$. 
We use  $\lambda_i(K, \Lambda) \  (i= 1, 2, \ldots, d+1)$ to denote   the $i$-th minimum  of $\Lambda$ with respect to $K$, i.e.~the
infimum of those numbers $\lambda$ such that $\lambda K\cap \Lambda$ contains $i$ linearly independent
vectors. 
We remark  here that $\lambda_{d+1}(K ,\Lambda)=\infty$. 
Let $\vol(\cdot)$ be the Lebesgue measure on $\R^d$. The covolume of $\Lambda$, denoted by
$\cov(\Lambda)$, is the Lebesgue measure of a fundamental domain of $\Lambda$. Write
\[
\theta (K, \Lambda):=\frac{\vol(K)}{\cov(\Lambda)}.
\]
By Minkowski's (second) theorem (see \cite{cassels}) one has 
\begin{align}\label{eq;minkowski}
\frac{2^d}{d!}\le \lambda_1(K, \Lambda)\cdots \lambda_d(K, \Lambda) \cdot \theta(K, \Lambda)\le 2^d.
\end{align}

For an affine subspace $H$ of $\R^d$ we let  $\vol_H(\cdot)$ be the 
Lebesgue measure on $H$ with respect to the subspace  Riemannian structure. 
To simplify the notation we let
$\vol_H(S)=\vol_H(S\cap H)$ for a Borel measurable  subset   $S$ of $\R^d$.
A subspace $H$ of $\R^d$ is said to be $\Lambda$-rational if 
$H\cap \Lambda$ is a lattice of $H$. The covolume of the lattice $H\cap \Lambda$ in $H$ is denoted by
$\cov_H(\Lambda)$.  
The same notations are used for the dual vector space
 $ \mathcal E_d^* $ (the vector space of linear functionals on $\R^d$) with respect to the standard Euclidean structure. 
 For every  $\varphi\in \mathcal E_d^*$, denote $H_\varphi=\ker \varphi$.

We use  $\|\cdot\|$ for the Euclidean norms on $\R^d$ and $ \mathcal E_d^* $. 
For a normed vector  space $V$  we use $B_r(V)$ (or $B_r$ if  $V=\R^d$) to denote the ball of radius $r$ centered at $0\in V$.
We will also use $K$-norms on $\R^d$ and $ \mathcal E_d^* $ defined by 
\begin{align}\label{eq;knorm}
\left\{   
\begin{array}{ll}
\|\bv\|_K=\inf \{r>0: \bv\in rK \}&
 \quad \bv\in \R^d \\
 \|\varphi\|_K=\sup_{\bv\in K}|\varphi(\bv)| &\quad  \varphi\in \mathcal E_d^* 
\end{array}.
\right.
\end{align}
It can be checked that $K$-norms satisfy the triangle inequality and other
axioms of  norm on a real vector space.

Let $\mathcal L_d$ be the space  of unimodular lattices in $\R^d$. 
The group $SL_d(\R)$ acts transitively  on $\mathcal L_d$ via $g\Lambda=\{g\bv:\bv\in \Lambda\}$. The stabilizer of 
$\Z^d $ is $SL_d(\Z)$, so we can identify $\mathcal L_d$ with $SL_d(\R)/SL_d(\Z)$ as 
topological spaces. 
For $g\in SL_d(\R)$ we let $g^*$ be the adjoint action on $ \mathcal E_d^* $ defined by 
$\varphi\to \varphi\circ g$. Note that with respect to the standard basis $\be_1, \ldots, \be_d$ on $\R^d$ and its dual basis $\be_1^*, \ldots, \be_d^*$ on  $ \mathcal E_d^* $, the matrix
$g^*$ is the transpose of $g$.
We define 
\begin{align*}
\mathcal K_\varepsilon(d)&=\{ \Lambda\in \mathcal L_d:\|\bv\|\ge \varepsilon, \ \forall\  \bv\in \Lambda\setminus \{0\} \} 
=\{ \Lambda\in \mathcal L_d:\lambda_1(B_1, \Lambda)\ge \veps\}.
\end{align*}

 The dual lattice of $\Lambda$ is 
 the lattice in $ \mathcal E_d^* $
 defined by 
 \[
 \Lambda^*=\{\varphi\in  \mathcal E_d^* : \varphi(\bv)\in \Z, \ \forall\  \bv\in\Lambda   \}.
 \]
 We also define 
 \begin{align}\label{eq;mel} 
 \mathcal K_\varepsilon^*(d)&=\{ \Lambda\in \mathcal L_d:\|\varphi\|\ge \varepsilon, \ \forall\  \varphi\in \Lambda^*\setminus \{0\} \}.
 \end{align}
Recall that $ \mathcal E_d^* $ can be naturally identified with $\wedge^{d-1}_\R \R^d$ with the standard  Euclidean structure. 
Under this identification one has $\Lambda^*=\wedge ^{d-1}_\Z \Lambda$.
Therefore,  $\mathcal K_\varepsilon^*$ is the set of $\Lambda\in \mathcal L_d$ with the property that each $\Lambda$-rational 
hyperplane intersects $\Lambda$ in a lattice of covolume greater than or equal to $\varepsilon$. 
Using the natural identification $\mathcal E_d^{**}=\R^d$ and (\ref{eq;minkowski})
one has
\begin{align*}
\Lambda\in \mathcal K^*_\varepsilon(d) \implies \lambda_1(B_1, \Lambda)\gg_d \varepsilon^{d-1}\quad 
\mbox{and}\quad 
\lambda_d(B_1, \Lambda)\ll_d  \varepsilon^{-1}.
\end{align*}
There is a basis $\bv_1, \ldots, \bv_d$ of $\Lambda$ with the properties
\[
\lambda_i(B_1, \Lambda)\le \|\bv_i\|\le 2^i \lambda_i(B_1, \Lambda)\quad (1\le i\le d),
\]
see \cite[Lemma X.6.2]{siegel}. This basis is called a Minkowski reduced basis of $\Lambda$.

A nonzero  vector $\bv\in \Lambda$ is said to be primitive if $\frac{1}{n}\bv\not\in \Lambda$
for all $n\in \N$. The set of primitive vectors in $\Lambda$ is denoted by $\widehat \Lambda$.

\begin{lem}\label{lem;primitive}
	Let $d\ge 2$.
	For every lattice 
	 $\Lambda$  of $\R^d$ and   
	 every bounded  centrally symmetric convex subset  $K$ of $ \R^d$ with $\lambda_d(K, \Lambda)\le 1$ we have 
	\begin{align*}
	 \sharp K \cap \widehat\Lambda= \big(\zeta(d)^{-1}+\eta(K, \Lambda )\big)\cdot\theta(K, \Lambda)
	\end{align*}
	where $\zeta$ is the Riemann $\zeta$-function and 
	\[
	|\eta(K, \Lambda)|	\ll_d \lambda_d(K, \Lambda)- \lambda_d(K,\Lambda)\log\lambda_1(K,\Lambda).
	\]
\end{lem}

Note that 
$\lambda_d(K, \Lambda)\le 1$ implies that the interior  of $K$ (denoted by $K^\circ$)  is nonempty.
To  prove Lemma \ref{lem;primitive} we need  a few  preparations (Lemmas \ref{lem;count 1}-\ref{lem;count 2}).

\begin{lem}\label{lem;count 1}
Let $d\ge 1$. For every lattice 
$\Lambda$  of $\R^d$ and   
every bounded  centrally symmetric convex subset  $K$ of $ \R^d$ with $\lambda_d(K, \Lambda)\le 1$	
one has 
\[
\sharp K\cap (\Lambda\setminus \{0\})=\big(1+\alpha(K, \Lambda)\big)\cdot\theta(K, \Lambda), 
\]
where $|\alpha(K, \Lambda)|\ll_d \lambda_d(K, \Lambda)$. 
\end{lem}

\begin{proof}
	It follows from the definition of successive minima that there is a set of  linearly independent vectors 
	$\{\bv_1, \ldots, \bv_d\}\subset\Lambda$ such that $\|\bv_i\|_K= \lambda_i(K, \Lambda)$. 
	So there is a  fundamental domain  $\Omega$
	of $\Lambda$ contained in 
	\[
	\{s_1 \bv_1+\cdots+s_d\bv_d: |s_i|\le 1/2 \}\subset   d\lambda_d(K, \Lambda) K. 
	\]
	It follows that 
	\begin{align}\label{eq;lambda K}
	\lambda_d(K, \Lambda)^d \theta(K, \Lambda)\gg_d 1.
	\end{align}

    Since $(\sharp K\cap \Lambda ) \cdot\cov({\Lambda}) = \vol (K\cap \Lambda +\Omega)$ 
    and $ K\cap \Lambda +\Omega\subset  (1+d\lambda_d(K, \Lambda) )K $, one has
	\begin{align}\label{eq;count 1.1}
	 (\sharp K\cap \Lambda  )\cdot \cov({\Lambda}) \le  \vol\big ((1+d\lambda_d(K, \Lambda) )K\big).
	\end{align}
	If $\lambda_d(K, \Lambda)\ge  1/d$, then the conclusion of the lemma  follows from (\ref{eq;count 1.1}).
	In the case where $\lambda_d(K, \Lambda)< 1/d$
     one has 
	$(1-d\lambda_d(K, \Lambda) )K   \subset K\cap \Lambda +\Omega$ which implies 
	\begin{align}\label{eq;count 1.2}
	\vol ((1-d\lambda_d(K, \Lambda) )K)\le  (\sharp K\cap \Lambda)  \cdot \cov({\Lambda}). 
	\end{align}
	In view of (\ref{eq;lambda K}) (which takes  care of  $0\in \Lambda$), (\ref{eq;count 1.1}) and (\ref{eq;count 1.2}) the conclusion of the lemma  
	also holds  in this case.  
\end{proof}

\begin{lem}\label{cor;count 1}
	 	Let $K$ and $\Lambda$ be as in Lemma \ref{lem;count 1}. Then
	\begin{align}
	  &\sharp K\cap \Lambda \asymp_d \theta(K, \Lambda).
	   \label{eq;cor count 1}
	\end{align}
\end{lem}
\begin{proof}
It  is clear from (\ref{eq;count 1.1}) that $\sharp K\cap \Lambda \ll_d \theta(K, \Lambda) $. 
On the other hand, it follows from \cite[Chapter 3 Theorem II]{cassels} that 
\begin{align*}
\sharp K\cap \Lambda\ge 2^{-d}\theta(K, \Lambda).
\end{align*}

\end{proof}

\begin{lem}\label{lem;teaching}
Let $\Lambda$ be a lattice of $\R^d$ ($d\ge 1$) and $K$ be a bounded centrally symmetric convex subset of $\R^d$ with nonempty interior. Then 
\begin{align}
\sharp K^\circ \cap \Lambda \asymp_d \sharp K \cap \Lambda \asymp_d \sharp  \overline K \cap \Lambda.
\label{eq;cor count 2}
\end{align}
\end{lem}
\begin{proof}
	Let $H$ be the linear span of $ \overline K\cap \Lambda$.
	Suppose $H$ is an $i$-dimensional real vector space and we assume 
	without loss of generality that $i>0$. 
	Since  an open neighborhood of $0$ is contained in  $K$ and $K$ is convex, 
	the closure of  $K^\circ \cap H$ is equal to  $\overline K\cap H$. 
	It follows from (\ref{eq;cor count 1}) that 
	\begin{align*}
	\sharp \overline K\cap \Lambda \asymp _i\frac{\vol_H (\overline K  )}{\cov_H(\Lambda )} 
	\quad \mbox{and }\quad 
	\sharp  K^\circ \cap \Lambda  \asymp _i \frac{\vol_H ( K^\circ  )}{\cov_H(\Lambda )} .
	\end{align*}
	Since  every convex subset of the Euclidean space has negligible boundary, one has 
	$\vol_H(\overline  K)=\vol_H( K^\circ)$. 
	Since there is no difference between a constant depending on $d$ and on   $\{1, \ldots, d\}$,  
	(\ref{eq;cor count 2}) holds.
\end{proof}

\begin{lem}\label{lem;count 2}
Let $K$ and $\Lambda$ be as in Lemma \ref{lem;count 1}. Suppose $\lambda_i(K, \Lambda)\le s\le s'\le \lambda_{j+1}(K, \Lambda)$
where $1\le i\le j\le d$, 
then 
\begin{align}\label{eq;count 2}
\left (\frac{s'}{s}\right)^i\ll_d
\frac{\sharp s' K\cap \Lambda}{\sharp sK \cap \Lambda }
\ll_d \left (\frac{s'}{s}\right)^j.
\end{align}
\end{lem}

\begin{proof}
	If  $\lambda_i(K, \Lambda)\le s\le s'\le \lambda_{i+1}(K, \Lambda)$, then 
	(\ref{eq;count 2})  follows from 
    Lemmas \ref{cor;count 1} and \ref{lem;teaching}. The general case  follows from  this special case. 
\end{proof}

\begin{proof}[Proof of Lemma \ref{lem;primitive}]
	The set \begin{align*}
	\widehat \Lambda=\Lambda\setminus\cup_{p\text{ prime}}p\Lambda=(\Lambda\setminus \{0\})\setminus
	\left( \cup_{p\text{ prime}}p\Lambda\setminus \{0\} \right ).
	\end{align*}
	Let  $\mu$ be the M\"obius function and $\lambda_i=\lambda_i(K,\Lambda)$.	
	It follows from the inclusion-exclusion principle that 
	\begin{equation}\label{eq;count 2.1}
	\begin{aligned}
	&\sharp K\cap (\widehat\Lambda\setminus \{0 \}) \\
	=&\sharp K\cap (\Lambda\setminus \{ 0\})+
	\sum_{\substack{p_1<\cdots <p_k\\\text{primes}}}(-1)^k\sharp K\cap(   p_1\cdots p_k\Lambda\setminus \{0\} )  \\
	=&\sum_{n=1}^\infty\mu(n)\sharp K\cap( n \Lambda\setminus \{0\})\\
	=&\sum_{n=1}^{\lfloor\lambda_d^{-1}\rfloor}\mu(n)\frac{\theta(K, \Lambda)}{n^d }\left(1+\alpha(K, n\Lambda)\right)
	+ \sum_{n=\lfloor\lambda_d^{-1}+1\rfloor}^{\lfloor\lambda_1^{-1}\rfloor}\sharp K\cap( n \Lambda\setminus \{0\}),
	\end{aligned}
	\end{equation}
	where in the last equality we use the notation of  Lemma \ref{lem;count 1} and the observation that for any integer $n> \lambda_1^{-1}$
	one has $K\cap n\Lambda=\{0\}$.

	It follows from  Lemmas \ref{cor;count 1} and  \ref{lem;count 2} (with $\lambda_1\le 
	s=n^{-1}\le s'= \lambda_d \le \lambda_{d+1}$) that 
	\begin{align}\label{eq;count 2.2}
	\sharp K\cap n\Lambda= \sharp n^{-1}K\cap \Lambda\ll_d \lambda_d^{d-1}n^{-1}\theta(K, \Lambda)
	\end{align}
	for $n\in [\lambda_d^{-1}, \lambda_1^{-1}]$.
	Using (\ref{eq;count 2.1}), (\ref{eq;count 2.2}) and the estimate of $\alpha$ in Lemma \ref{lem;count 1}, we
	can write 
	\[
	\sharp \widehat\Lambda\cap K= (\zeta(d)^{-1}+\eta(K, \Lambda))\theta(K, \Lambda),
	\]
	where 
	\begin{align*}
	|\eta(K, \Lambda)|&\ll_d \sum_{n=\lfloor\lambda_d^{-1}+1\rfloor}^\infty \frac{1}{n^d}+
	\sum_{n=1}^{\lfloor\lambda_d^{-1}\rfloor}\frac{1}{n^d}|\alpha(K, n\Lambda)|+
	\sum_{n=\lfloor\lambda_d^{-1}+1\rfloor}^{\lfloor\lambda_1^{-1}\rfloor}\lambda_d^{d-1}n^{-1} \\
	&\ll_d \lambda_d^{d-1}+\sum_{n=1}^{\lfloor\lambda_d^{-1}\rfloor} \frac{\lambda_d}{n^{d-1}}+\lambda_d^{d-1}\log (\lambda_d \lambda_1^{-1}) \\
	& \ll  \lambda_d -\lambda_d\log  \lambda_1.
	\end{align*}

\end{proof}


\begin{lem}\label{lem;hyperplane}
	Let $K$ be a bounded  centrally symmetric convex  subset of $\R^d$ $(d\ge 2)$ with nonempty interior and let $\varphi\in  \mathcal E_d^* \setminus \{0\}$. 
	Then 
	$
	\vol_{H_\varphi}(K)\asymp_d \frac{\|\varphi\| \vol(K) }{\|\varphi\|_K}.
	$
\end{lem}

\begin{proof}
	Since $\|\varphi\|/\|\varphi\|_K=\|t\varphi\|/\|t\varphi\|_K  $ for every nonzero real number $t$, 
	we assume  without loss of generality that $\|\varphi\|=1$. 
	Using Fubini's theorem one has 
	\begin{align}\label{eq;vol decom}
	\vol(K)=\int_{\R} \vol_{\varphi^{-1}(t)}(K)\dd t=2\int_{0}^{\|\varphi\|_K} \vol_{\varphi^{-1}(t)}(K)\dd t.
	\end{align}
	Since $K$ is centrally symmetric,  for each $t\in \R$ and $\bv\in K$ with $\varphi(\bv)=t$ one has
	$-\bv\in K$ and the line segment  joining $-\bv$ and $\varphi^{-1}(t)$ is in $K$. Therefore 
	\[
	\vol_{\varphi^{-1}(t)}(K)\le 2 \vol_{H_\varphi}(K).
	\] 
	This estimate  and (\ref{eq;vol decom}) imply   
	\begin{align}\label{eq;hyperplane 1}
	\vol(K)\le 4  \|\varphi\|_K \vol_{H_\varphi}(K) .
	\end{align}

	For any   $0<\varepsilon< \|\varphi\|_K$, there is a vector $\bv\in K$ such that 
	$	
	\varphi(\bv)= \|\varphi\|_K-\varepsilon.
	 $
	Since $K$ is   convex, for every $t$ with  $0< t< \|\varphi\|_K-\varepsilon$ one has
	\[
	K\cap \varphi^{-1}(t)\supset \frac{\|\varphi \|_K-\varepsilon-t}{\|\varphi\|_K-\varepsilon} (K\cap H_\varphi)+ 
	\frac{t}{\|\varphi\|_K-\varepsilon}\bv,
	\]
	which implies 
	\[
	\vol_{\varphi^{-1}(t)}(K)\ge  \left( \frac{\|\varphi\|_K-\varepsilon-t}{\|\varphi\|_K-\varepsilon}\right)^{d-1}\vol_{H_\varphi}(K ) . 
	\]
	This estimate and  (\ref{eq;vol decom}) imply 
	\begin{align*}
	\vol(K)& \ge 2\int_{0}^{\|\varphi\|_K-\varepsilon} \vol_{\varphi^{-1}(t)}(K)  \dd t  \\
	&\ge  2 \vol_{H_\varphi}(K )\int_{0}^{\|\varphi\|_K-\varepsilon}  \left( \frac{\|\varphi\|_K-\varepsilon-t}{\|\varphi\|_K-\varepsilon}\right)^{d-1}\dd t  \\
	& =2\vol _{H_\varphi}(K) \frac{\|\varphi\|_K-\varepsilon}{d}  . 
	\end{align*}
	Taking  $\varepsilon=\|\varphi\|_K/2$,  we have
	\begin{align}\label{eq;hyperplane 2}
	\vol(K)\ge \frac{\vol_{H\varphi}(K)\|\varphi\|_K}{d}.
	\end{align}
	The lemma follows from (\ref{eq;hyperplane 1}) and (\ref{eq;hyperplane 2}). 

\end{proof}

\subsection{Lattice points counting in $\R^3$}\label{sec;count r3}

In this subsection we estimate the number of the lattice points in 
\begin{align}\label{eq;m r}
M_\br&=\{(x_1, x_2, x_3)\in \R^3 : |x_i|\le r_i\}
\end{align}
for a triple of positive real numbers $\br=(r_1, r_2, r_3)$.
Lemmas \ref{lem;many vectors}, \ref{lem;application 1}
and \ref{lem;application 2} will be used in \S \ref{sec;refinement}
for lower bound estimation. 
The  two latter results are deduced from
  Lemma~\ref{lem;technical} where  we prove an upper bound of  the number of the lattice
points in  $M_\br$ that lie in certain badly shaped hyperplanes.
Let 
\begin{equation*}
\begin{aligned}
M_{\br}^*     &=\{\varphi\in \mathcal E_3^* : |x_i^\varphi|\le r_i\} .
\end{aligned}
\end{equation*}
where $x_1^\varphi, x_2^\varphi, x_3^\varphi\in \R$ are
the coordinates of $\varphi\in \mathcal E _3^*$, that is, 
$\varphi=x_1^\varphi\be_1^*+ x_2^\varphi\be_2^*+ x_3^\varphi\be_3^*$.

\begin{lem}\label{lem;many vectors}
	There exists a positive  real number $ \til c< 1$ such that for every lattice  $\Lambda$ of
	$\R^3$ and every triple of positive real numbers $\br$ with 
	\[
    \lambda_3(M_\br, \Lambda )\le \til c\quad \mbox{and}\quad 	
    -\lambda_3(M_\br, \Lambda )\log
  \lambda_1(M_\br, \Lambda ) \le \til c
	\]
	one has 
	\[
     \frac{4}{5\zeta(3)}\theta(M_\br, \Lambda)\le \sharp	M_{\br}\cap \widehat \Lambda \le \frac{6}{5\zeta(3)}\theta(M_\br, \Lambda).
	\]
\end{lem}
\begin{proof}
	This lemma follows directly from Lemma \ref{lem;primitive}. 
\end{proof}

Let us fix $0<s<1/2$ and $\br\in \R^3$ with $1\le r_1\le r_2$ and $r_3=1$. 
Let 
\[
\|\varphi\|_\br=\max \{r_1|x^\varphi_1| , r_2|x^\varphi_2| , | x^\varphi_3|\}. 
\]
It can be checked directly that 
\begin{align}\label{eq;r norm}
\|\varphi\|_\br\le  \|\varphi\|_{M_\br} \le 3\|\varphi\|_\br.
\end{align}
For $q>0$ let 
 \begin{align}\label{eq;jia}
  N_q(\br, s)=\{\varphi\in \mathcal E_3^*: |x^\varphi _1|\le s, |x^\varphi_2|\le s, \|\varphi\|_{\br} \le q \}.
 \end{align}
A direct calculation shows that 
 \begin{align*}
 N_q(\br, s)=\left\{
 \begin{array}{ll}
M^*_{\frac{q}{r_1}, \frac{q}{r_2}, q}\quad 
& q\le sr_1 \\               
M^*_{s, \frac{q}{r_2}, q}  \quad 
 & sr_1\le q\le sr_2\\
M^*_{s, s, q} \quad 
& sr_2\le q
 \end{array}.
 \right.
 \end{align*}
 For a lattice $\Lambda$ of $\R^3$ let $q_i(\Lambda, \br, s )\  (i=1,2,3)$  be the infimum of those
 positive  numbers $q$ such that $N_q(\br, s)\cap \Lambda^*$ contains 
 $i$ linearly independent vectors. 
We use $\widehat {\Lambda^*} $ to denote the set of primitive vectors of 
 the dual lattice
 $\Lambda^*$.
  In the next lemma  we give an upper bound of the cardinality of 
 	\begin{align}\label{eq;jia1}
 	\mathcal S(\Lambda, \br, s):=\{\bv\in  M_\br \cap \widehat \Lambda: \varphi(\bv)=0 \mbox{ for some } \varphi\in N_{3sr_2}(\br, s)\cap \widehat{\Lambda^*} \}. 
 	\end{align}
 \begin{lem}\label{lem;technical}
Let $\Lambda$ be a unimodular lattice of $\R^3$ with    $ q_1(\Lambda, \br, s)\ge s^{-2} $. The  following
statements  hold:
\begin{enumerate}[label=(\roman*)]
\item  if $r_1=r_2$ and  $ q_3(\Lambda, \br, s) \le 2 s^{-1/2} r_2  $ then 
$
\sharp\mathcal S(\Lambda, \br, s)\ll  s^{1/2}\cdot \vol (M_\br);
$
\item if $r_1< r_2$ and  $q_3(\Lambda,\br,s)\log q_3(\Lambda, \br,s)\le sr_2$ then 
$
\sharp\mathcal S(\Lambda, \br, s) \ll  s^2\cdot \vol (M_\br).
$
\end{enumerate}
 \end{lem}

 \begin{proof}
We write  $N_q=N_q(\br, s), q_i=q_i(\Lambda, \br, s)$ and $\mathcal S=\mathcal S(\Lambda, \br, s) $ for simplicity. 
If $N_{3sr_2}\cap \widehat {\Lambda^*}$ is empty, then there is nothing to prove. So we assume in the remaining of the  proof that 
$N_{3sr_2}\cap \widehat {\Lambda^*}\neq \emptyset$, i.e.~$q_1\le 3sr_2$. 
It is clear from the definition that 
\begin{align}\label{eq;hyperplane 0}
  \sharp \mathcal S\le \sum _{\varphi\in N_{3sr_2}\cap \widehat {\Lambda^*}}\sharp H_\varphi\cap M_\br \cap \widehat \Lambda.
\end{align}

We claim that  for all $\varphi\in N_{3sr_2}\cap \widehat {\Lambda^*}$ 
\begin{align}\label{eq;a factor}
\sharp  H_\varphi\cap M_\br \cap \widehat \Lambda \ll \frac{\vol(M_\br)}{\|\varphi\|_{M_\br}}\le \frac{\vol(M_\br)}{\|\varphi\|_{\br}}, 
\end{align}
where the second inequality follows from (\ref{eq;r norm}). 
If $\sharp  H_\varphi\cap M_\br \cap \widehat \Lambda >2$, then   (\ref{eq;cor count 1})  and Lemma \ref{lem;hyperplane} imply 
\[
\sharp  H_\varphi\cap M_\br \cap \widehat \Lambda \ll 
\frac{\vol_{H_\varphi}( M_\br)}{\cov_{H_\varphi} ( \Lambda)}\ll \frac{\|\varphi\| \vol(M_\br)}{\cov_{H_\varphi}(\Lambda)\|\varphi\|_{M_\br}}
=\frac{\vol(M_\br)}{\|\varphi\|_{M_\br}}.
\]
On the other hand  for  every $\varphi\in N_{3sr_2}\cap \widehat {\Lambda^*}$ we always have 
\[
\frac{\vol(M_\br)}{\|\varphi\|_{M_\br}}\ge \frac{8r_1r_2}{9sr_2}\gg 2.
\]
This completes the proof of the claim.

In view of  (\ref{eq;hyperplane 0}) and (\ref{eq;a factor})
 it suffices to estimate 
\begin{align}\label{eq;not many 2}
\eta: =\sum _{\varphi\in N_{3sr_2}\cap \widehat{\Lambda^*}} \|\varphi\|_{\br}^{-1}=
\frac{1}{3sr_2}\sharp N_{3sr_2}\cap \widehat {\Lambda^*}  +\sum_{\varphi\in N_{3sr_2}\cap \widehat{\Lambda^*}} \int^{3sr_2}_{\|\varphi\|_{\br}} 
\frac{1}{q^2}\dd q.
\end{align}
The second term of the right hand side of (\ref{eq;not many 2}) is
\begin{align*}
\eta_2:=\sum_{\varphi\in N_{3sr_2}\cap \widehat{\Lambda^*}} \int^{3sr_2}_{\|\varphi\|_{\br}} 
\frac{1}{q^2}\dd q=\sum_{\varphi\in N_{3sr_2} \cap \widehat{\Lambda^*}}\int_{q_1}^{3sr_2}
\frac{\mathbbm{1}_q(\|\varphi\|_{\br})}{q^2}\dd q,
\end{align*}
where $\mathbbm{1}_q$ is the indicator function on $\R$ defined by
$\mathbbm{1}_q (x)=1$ if $x\le q$ and $0$ otherwise. 
Using Fubini's theorem  one has 
\begin{align*}
\eta_2  &=\int_{q_1}^{3sr_2}
\sum_{\varphi\in N_{3sr_2}\cap \widehat{\Lambda^*}}\frac{\mathbbm{1}_q(\|\varphi\|_{\br})}{q^2}\dd q   \\
&\le \int_{q_1}^{3sr_2}\frac{\sharp N_q\cap \widehat{\Lambda^*}}{q^2}\dd q .\notag
\end{align*}
If  $q_1\le q< q_2$ then
 $\sharp N_q\cap \widehat {\Lambda^*}=2$.
So 
\begin{align}\label{eq;not many 10}
\int_{q_1}^{q_2} \frac{\sharp N_q\cap \widehat {\Lambda^*}}{q^2}\dd q \le \int_{q_1}^{q_2} \frac{2\dd q}{q^2} \le \frac{2}{q_1}, 
\le 2s^2
\end{align}
where in the last inequality we use the assumption $q_1\ge s^{-2}$.

From here we  consider  different cases according to the two situations in the statement of the lemma.

\noindent{\bf Case \rmnum{1}:} 
 we show $\eta\ll\sqrt s$ under the assumption of (\rmnum{1}).
 We first compute 
 \begin{align}\label{eq;not many 11}
 \int_{sr_2}^{3sr_2} \frac{\sharp N_q\cap \widehat {\Lambda^*}}{q^2}\dd q
 &\le \int_{sr_2}^{3sr_2} \frac{\sharp N_{3sr_2}\cap  \Lambda^*}{q^2} \dd q 
 \le \frac{\sharp N_{3sr_2}\cap  \Lambda^*}{sr_2}.  
 \end{align}
 Note that $q_3\le 2 s^{-1/2}r_2$ by assumption
 and $sr_2< s^{-1/2} r_2$ since $s<1$. 
 So by Lemma \ref{cor;count 1} (which is used in the second inequality below)
 \begin{align}\label{eq;not many 12}
 \frac{\sharp N_{3sr_2}\cap  \Lambda^*}{sr_2}&\le \frac{\sharp N_{3s^{-1/2}r_2}\cap  \Lambda^*}{sr_2} 
 \ll  \frac{\vol( N_{3s^{-1/2}r_2})}{sr_2}
 \ll   s^{1/2}. 
 \end{align}
 
\noindent{\bf Case \rmnum{1}.1:} suppose  $q_2>3sr_2$. 
It follows from (\ref{eq;not many 2}),  (\ref{eq;not many 10}), the assumption $q_1\ge s^{-2}$ and the observation
\[
\frac{1}{3sr_2}\sharp N_{3sr_2}\cap \widehat {\Lambda^*} \le \frac{2}{3sr_2}\le \frac{2}{q_1}
\]
that $\eta\ll s^2\le \sqrt s$.

\noindent {\bf Case \rmnum{1}.2:}  suppose  $sr_2\le q_2\le 3sr_2$. 
It follows from  (\ref{eq;not many 2}),  (\ref{eq;not many 10}),   (\ref{eq;not many 11}) and (\ref{eq;not many 12}) that 
\begin{align}
\eta 
 \le 2 s^2+\frac{2}{sr_2} \sharp N_{3sr_2}\cap  \Lambda^*
\ll   s^{1/2}. \notag
\end{align}

\noindent {\bf Case \rmnum{1}.3:}
suppose  $q_2<sr_2$.  For all   $q_2< q< sr_2=sr_1$  one has 
\begin{align*}
\sharp  N_q\cap \widehat {\Lambda^*}=\sharp\left( \frac{q}{sr_2} N_{sr_2}\right)\cap \widehat{\Lambda^*}\le \sharp \left(\frac{q}{sr_2} N_{sr_2}\right)\cap \Lambda^*\ll  
 \left(\frac{q}{sr_2}\right)^2 \sharp N_{sr_2}\cap \Lambda^*, 
\end{align*}
where in the last inequality we use Lemma \ref{lem;count 2}. Using this estimate, $sr_2< s^{-1/2}r_2$ and Lemma \ref{cor;count 1}
we get
\begin{align*}
\sharp  N_q\cap \widehat {\Lambda^*}\ll   \left(\frac{q}{sr_2}\right)^2 \sharp N_{s^{-1/2}r_2}\cap \Lambda^*\ll \frac{q^2 \vol( N_{s^{-1/2}r_2})}{s^2r_2^2}\ll  \frac{q^2s^{1/2}}{sr_2}.
\end{align*}
So 
\begin{align}\label{eq;q1231}
\int_{q_2}^{sr_2} \frac{\sharp N_q\cap\widehat {\Lambda^*}}{q^2}\dd q 
\ll \int_{q_2}^{sr_2}\frac{ s^{1/2}\dd q}{ sr_2}
\le s^{1/2}.
\end{align}	
It follows from (\ref{eq;not many 2}),  (\ref{eq;not many 10}), (\ref{eq;not many 11}), (\ref{eq;not many 12}) and (\ref{eq;q1231})
that $\eta \ll \sqrt s$.

\noindent {\bf Case \rmnum{2}: } we show $\eta\ll s^2$  under the assumption of (\rmnum{2}). 
Since $s<1/2$ and $q_1\ge s^{-2}$, we have $q_3\ge q_1\ge 4$. 
Therefore $q_3< sr_2$, since 
 $q_3\log q_3\le sr_2$ by the assumption. 
 So by Lemma \ref{cor;count 1}
\begin{align}\label{eq;not many 13}
\frac{\sharp N_{3sr_2}\cap \widehat {\Lambda^*}}{3sr_2}\ll \frac{\vol(N_{3sr_2})}{3sr_2} =8 s^2.
\end{align}
Similarly, for   $sr_2< q< 3sr_2$ one has
$
\sharp  N_q\cap \widehat {\Lambda^*}\le \sharp  N_{q}\cap \Lambda^*\ll \vol(N_q) \le 8s^2 q
$.
So 
\begin{align}\label{eq;not many 14}
\int_{sr_2}^{3sr_2} \frac{\sharp N_q\cap\widehat {\Lambda^*}}{q^2}\dd q 
\ll \int_{sr_2}^{3sr_2}\frac{8s^2 \dd q}{q}
=8 s^2 \log 3.
\end{align}	
Using Lemma \ref{cor;count 1} for  $q_3< q<sr_2$, one has 
$
\sharp N_q \cap \widehat {\Lambda^*}\ll \vol(N_q)
 \ll q^2s/r_2
$.
It follows that 
\begin{align}\label{eq;not many 15}
\int_{q_3}^{sr_2}\frac{\sharp N_q \cap \widehat {\Lambda^*}}{q^2}\dd q
\ll \frac{s(sr_2-q_3)}{r_2}\le s^2.
\end{align}

Now we estimate $\sharp N_q\cap \widehat {\Lambda^*}$ for $q_2\le  q<  q_3<sr_2$. 
Let $H$ be the $\R$-linear span of $N_{q_2}\cap \Lambda^*$. 
We claim that \[
\mathrm{vol}_H(N_q)\le \frac{q}{q_3}\mathrm{vol}_H(N_{q_3}).
\]
If $H=\mathrm{span}_{\R}\{\be_2^*,\be_3^* \}$, then the claim follows easily from the definition  
of $N_q$ for $q<sr_2$. Otherwise, the intersection of $H$  with  the affine hyperplane 
$H_t=t \be_1^*+ \mathrm{span}_{\R}\{\be_2^*,\be_3^* \}$ is a line. 
The  length of $H\cap H_t \cap N_q  $    is  at most  $\frac{q}{q_3}$ times  the length 
of $H\cap H_t \cap N_{q_3}  $. Since the volume of $H\cap N_q$ is proportional to  the integration of    the length of $H\cap H_t \cap N_q  $ with respect to  $t$, the claim follows. 
It follows from the claim that for  $q_2\le  q<   q_3$ 
\begin{align}\label{eq;maobi0}
\sharp  N_q\cap \widehat {\Lambda^*}\ll \frac{\vol_H(N_q)}{\cov_H(\Lambda^*)}
\le   \frac{q}{q_3}  \frac{\vol_H(N_{q_3})}{\cov_H(\Lambda^*)}.
\end{align}  
Note that 
\begin{equation}\label{eq;maobi}
\begin{aligned}
 \frac{\vol_H(N_{q_3})}{\cov_H(\Lambda^*)}&\ll \sharp N_{q_3}\cap H\cap \Lambda^* 
&& \qquad \mbox{by (\ref{eq;cor count 1})}\\
             & \ll \sharp N_{q_3}^{\circ}\cap H\cap \Lambda^*   &&\qquad \mbox{by (\ref{eq;cor count 2})}\\
             & =\sharp N_{q_3}^{\circ}\cap \Lambda^* \\
             & \le\sharp  N_{q_3 }\cap \Lambda^*.
\end{aligned}
  \end{equation}
In view of (\ref{eq;maobi0}) and (\ref{eq;maobi}), for $q_2\le q< q_3$ we have 
\begin{align*}  
\sharp  N_q\cap \widehat {\Lambda^*} \ll \frac{q}{q_3}  \sharp N_{q_3}\cap \Lambda^*\ll \frac{q}{q_3}\vol(N_{q_3})
\ll \frac{q_3q s}{r_2}.
\end{align*}
Therefore 
\begin{align}\label{eq;not many 16}
\int_{q_2}^{q_3}\frac{\sharp  N_q\cap \widehat {\Lambda^*}}{q^2}\dd q\ll s^2 \frac{q_3 \log q_3}{sr_2}\le s^2,
\end{align}
where in the last inequality we use the assumption $q_3\log q_3\le sr_2$.
It follows from (\ref{eq;not many 2}),  (\ref{eq;not many 10}),
(\ref{eq;not many 13}), (\ref{eq;not many 14}), (\ref{eq;not many 15}) and   (\ref{eq;not many 16}) that  $\eta\ll s^2$.

\end{proof}

We will   apply Lemma \ref{lem;technical} in the  two concrete cases below. 
We first introduce some notation. 
Let $w=(w_1,w_2)$ be as in Theorem \ref{thm;main}. We moreover assume that $w_1>w_2$.  
We fix $C\ge 1$ such that $C$ is an implied constant for 
the conclusions of Lemma \ref{lem;technical} (\rmnum{1}) and (\rmnum{2}).
For a lattice $\Lambda'$ of $\R^3$ and fixed $\br, s$ we let 
$q_i(\Lambda')=q_i(\Lambda', \br, s)$ $(i=1,2,3)$ for simplicity.
Similarly, we write $N_q=N_q(\br, s)$. Recall that 
 \begin{align*}
	\mathcal L_3'=\{\Lambda\in \mathcal L_3: \Lambda\cap \R\be_3=r \Z\be_3 \mbox{ for  some } r 
	\mbox{ with } 1/2<r\le 1\}.
\end{align*}

\begin{lem}\label{lem;application 1}
	Let $s=\varepsilon^2, \br=(r_1, r_2, r_3)=(\varepsilon e^{t}, \varepsilon e^{t}, 1)$,
	 $\Lambda\in \mathcal K^*_{\varepsilon^2}\cap  \mathcal L_3'$ and 
	 $a_t=\mathrm{diag}(e^{w_1t}, e^{w_2t}, e^{-t})$.
	 There exists a positive real number  $ \til \varepsilon\le 1 $  such that for all 
	 $\varepsilon ,t> 0$ with  $e^{-w_2 t/20}<\varepsilon<\til \varepsilon$
	   one has 
	 \begin{align}\label{eq;application1.1}
	 \sharp\mathcal S(a_t\Lambda, \br, s)\le  \varepsilon^{1/2}\cdot \vol (M_\br).
	 \end{align}
\end{lem}

\begin{proof}
	 
	We will show that the lemma holds for 
	$\til \varepsilon =\frac{1}{100 C^2}$.
	 In view of Lemma \ref{lem;technical} (\rmnum{1}) and the choice of $\til \varepsilon$, it suffices to prove 
	\[
	  q_1(a{_t} \Lambda)\ge s^{-2}\quad \mbox{ and}\quad  q_3(a_t\Lambda)\le 2 s^{-1/2}r_2.
	\]
	We need to look at the lattice points in 
	\begin{align}\label{eq;hei0}
	N_q\cap (a_t\Lambda)^*=N_q \cap a_{-t}^*\Lambda^*= a^*_{-t}(a^*_{t} N_q\cap \Lambda^*),
	\end{align}
	where $a^*_t$ is the transpose  of $ a_t $ with respect to the standard basis of $\R^3$ and its dual basis. 
	
	Since 
	$e^{-w_2 t/20}<\varepsilon$ by assumption, we have
	\begin{align}\label{eq;bigger than}
	\varepsilon^{20} e^{w_2 t}>1.
	\end{align}
	It follows that   $s^{-2}\le r_1s$. So 
	\[
	N_{s^{-2}}=\{\varphi\in \mathcal E_3^*:
	|x^\varphi_1|\le  \varepsilon^{-5} e^{-t},	|x^\varphi_2| \le \varepsilon^{-5} e^{-t}, 	|x^\varphi_3|\le \varepsilon^{-4}      \}.
	\]
	Hence
	\begin{align*}
	a^*_{t} N_{s^{-2}}=
	\{\varphi\in \mathcal E_3^*:
		|x^\varphi_1|\le  \varepsilon^{-5} e^{-w_2t},	|x^\varphi_2| \le \varepsilon^{-5} e^{-w_1t}, 	|x^\varphi_3|\le \varepsilon^{-4}e^{-t}\}  
	\end{align*}
	which is contained in the interior of  	 $B_{\varepsilon^2}({\mathcal E_3^*})$ by (\ref{eq;bigger than}).
	Since $\Lambda\in \mathcal K_{\varepsilon^2} ^*$,  the dual lattice $\Lambda^*$ has no nonzero 
	vectors in $a^* _{t}  N_{s^{-2}}$.
	Therefore 
	$q_1( a_t \Lambda)>s^{-2}$.

    Next we turn to the proof of $q_3(a_t\Lambda)\le 2s^{-1/2}r_2=2e^t$. 
    Since $2e^t>r_2>r_2s$ one has 
	\begin{align}\label{eq;hei}
	a^*_{t} N_{2e^{t}}=
	\{\varphi\in  \mathcal E_3^* :
	|x^\varphi_1|\le s  e^{w_1t},|x^\varphi_2| \le s e^{w_2t},|x^\varphi_3|\le 2\}  .
	\end{align}
	It follows from the assumption  $\Lambda\in\mathcal L_3'$ that  $r\be_3\in \widehat\Lambda$ for  some $r$ with $1/2<r\le 1$. 
	Let $\mathbf{pr}: \R^3\to \R^2$ be the orthogonal  projection to the subspace $\R\be _1+\R\be_2$. 
	It follows that $\mathbf{pr}(\Lambda)$ is a lattice  with covolume $1/r$ in $\R^2$.
	Suppose  $\bv\in \Lambda$ and $ \|\mathbf{pr}(\bv)\|= \lambda_1(B_1({\R^2}), \mathbf{pr}(\Lambda))$,
	then 
	\begin{align}\label{eq;zimmer}
	\lambda_1(B_1({\R^2}), \mathbf{pr}(\Lambda))\ge  r \cdot \lambda_1(B_1({\R^2}), \mathbf{pr}(\Lambda))=\|\bv \wedge r \be_3 \|\ge \varepsilon^2, 
	\end{align}
	where in the last inequality we use  $\Lambda\in \mathcal K^*_{\varepsilon^2}$.
	 It follows from 
	  of (\ref{eq;minkowski}) and (\ref{eq;zimmer}) that 
	\[
	\lambda_2(B_1({\R^2}), \mathbf{pr}(\Lambda))\le 8 \varepsilon^{-2} .
	\] 
	Recall that there exists a Minkowski reduced basis $v^{(1)}, v^{(2)}$ of 
	$\mathbf{pr}(\Lambda)$ with the property 
	$\|v^{(1)}\|\le 4 \lambda_1(B_1({\R^2}), \mathbf{pr}(\Lambda))$ and 
	$\|v^{(2)}\|\le 4 \lambda_2(B_1({\R^2}), \mathbf{pr}(\Lambda))$. 
	Let $\bv_i\in \Lambda\ (i=1,2)$ with $\mathbf{pr}(\bv_i)=v^{(i)}$ and 
	$\be_3^*(\bv_i) < 1$. 
	It follows that $\bv_1, \bv_2,\bv_3:= r\be_3$ is a basis 
	of $\Lambda$. Recall that $\Lambda^*$ can be identified with $\wedge^2\Lambda$ as 
	 Euclidean spaces.
	 In view of (\ref{eq;hei0}) and (\ref{eq;hei})
	 it suffices to show that  the coordinates of  
	 \[
	 \bv_i\wedge \bv_j= x_1 \be_2 \wedge\be_3+ x_2\be_1 \wedge\be_3+   x_3 \be_1\wedge \be_2
	 \]
	 satisfy
	 \begin{align}\label{eq;eating}
	 |x_1|\le se^{w_1t},\  |x_2|\le se^{w_2t} \quad \mbox{and }|x_3|\le 2.
	 \end{align}
	 It follows from the definition of $\bv_i$ that 
	 $\|\bv_i\|\le 1+32\varepsilon^{-2}\le 33\varepsilon^{-2}$. So 
	  $$\|\bv_i\wedge \bv_j\|\le \|\bv_i\|\cdot \|\bv_j\|\le  33^2 \varepsilon^{-4}\le  se^{w_2t}$$
	  where in the last inequality we use 
	  (\ref{eq;bigger than}) and the assumption $\varepsilon <\widetilde \varepsilon$.   Therefore     the
	  upper bounds of $|x_1|$ and $|x_2|$ in (\ref{eq;eating}) hold. 
	Finally note that $x_3=0$ unless $\{i, j \}=\{1, 2\}$ where $|x_3|=1/r\le 2$.

\end{proof}

\begin{lem}\label{lem;application 2}
	Let $ \br=(r_1, r_2, r_3)=(\varepsilon e^{w_1 t}, \varepsilon e^{(w_1+2w_2)t}, 1)$,
	$\Lambda\in \mathcal K^*_{\varepsilon^2}$ and  
	$b_t=\mathrm{diag}(e^{(w_1-w_2)t}, e^{2w_2t}, e^{-t})$.
	Then there exists a positive real number  $\til s\le 1 $ such that for all 
	$s>0$ and $t>0$ with 
	$ e^{-\delta t}< \varepsilon<s <\til s$
	where $\delta=\frac{1}{20}\min\{{w_2}, {w_1-w_2}\}$
	 one has 
	\begin{align}\label{eq;application 2.1}
	\sharp\mathcal S(b_t\Lambda, \br, s)\le  s\cdot \vol (M_\br).
	\end{align}
	 
\end{lem}
\begin{proof}
There exists a positive real number  $c<1$ such that if $e^{-\delta t}< c$ then 
\begin{align}\label{eq;no heat 1}
e^{w_2t/20}\ge (1+\frac{w_2}{2})t.
\end{align}
We will  show that 
(\ref{eq;application 2.1}) holds
for $\til s =\min\{\frac{1}{100 C} , c \}$.
   In view of Lemma \ref{lem;technical} (\rmnum{2}) it suffices to prove  
   \[
	q_1(b_t \Lambda)\ge s^{-2}\quad  \mbox{and} \quad q_3(b_t\Lambda)\log q_3(b_t\Lambda)\le s \varepsilon e^{(w_1+2w_2) t}.
	\]

	Using the  assumption $e^{-\delta t}<\varepsilon <s$ one has 
	\begin{align}\label{eq;no heat}
	s^5\varepsilon^5 \ge e^{-10\delta t},
	\end{align}
 which implies   $s^{-2}\le sr_1$.
It follows that 
\begin{align*}
 b_{t}^* N_{s^{-2}} 
=  \{ \varphi\in  \mathcal E_3^* : |x^\varphi_1| \le e^{-w_2t}\varepsilon^{-1}s^{-2},  
|x^\varphi_1| \le e^{-w_1t}\varepsilon^{-1}s^{-2}   ,|x^\varphi_1|\le  e^{-t}s^{-2} \}
\end{align*}
which in view of (\ref{eq;no heat}) is contained in the interior of  $ B_{\varepsilon^2}(\mathcal E _3^*)$.
Since $\Lambda\in \mathcal K_{\varepsilon^2}^*$, one has 
\[
N_{s^{-2}}\cap (b_t \Lambda)^*=b_{-t}^*\left( b_{t}^*N_{s^{-2}}\cap \Lambda^*\right)=\{0\}. 
\]
Therefore $q_1(b_t\Lambda)\ge s^{-2}$.

We claim  that $q_3(b_t\Lambda)\le e^{(1+w_2/2)t}$. 
Note that $r_1s< e^{(1+w_2/2)t}<r_2 s$ by  (\ref{eq;no heat}). 
Therefore 
\begin{align*}
 b_{t}^*N_{e^{(1+w_2/2)t}} 
=
\{\varphi\in  \mathcal E_3^* : |x^\varphi_1|\le  se^{(w_1-w_2)t}, 
|x^\varphi_2|\le  e^{3w_2t/2}\varepsilon^{-1}, |x^\varphi_3|\le e^{w_2t/2} \}.
\end{align*}
Since $\Lambda\in \mathcal K_{\varepsilon^2}^*$, 
Minkowski's second theorem (\ref{eq;minkowski}) implies $\lambda_3(B_1({ \mathcal E_3^* }),\Lambda^*)\le {2\varepsilon^{-4}} $.
Therefore there exists Minkowski reduced basis 
$\varphi_1, \varphi_2, \varphi_3$ of $\Lambda^*$  such that 
$\|\varphi_i\|\le {16} \varepsilon^{-4}\le \varepsilon^{-5} $. 
Using  (\ref{eq;no heat}) it is not hard to see that  
$\varphi_1, \varphi_2, \varphi_3\in b_{t}^*N_{e^{(1+w_2/2)t}}$.
Therefore $q_3(b_t\Lambda)\le e^{(1+w_2/2)t}$. 
This completes the proof of the claim.
Finally we have
\begin{align*}
q_3(b_t\Lambda)\log q_3(b_t\Lambda)&\le e^{(1+w_2/2)t}(1+w_2/2)t && \mbox{by the claim} \\
& \le e^{(1+w_2/2)t} e^{w_2 t/20 }  && \mbox{by (\ref{eq;no heat 1})}\\
& \le s\varepsilon e^{(w_1+2w_2)t} && \mbox{by (\ref{eq;no heat})}.
\end{align*}
\end{proof}

\section{Lower Bound}

Recall that  $\mathcal L_3$ is the space  of unimodular lattices in $\R^3$ and $\N=\{1, 2,   3, 4 ,\ldots   \}$. 
Let $a_t $ and $h(x)$ be as in (\ref{eq;a t}) and (\ref{eq;h x new}) respectively. 
A vector $x\in \R^2$ is $w$-singular if and only if the trajectory 
$
\{ a_t  h(x)\Z^3 : t\ge 0\} 
$
is divergent, i.e.~for any compact subset $\mathcal K$ of $\mathcal L_3$, there exists $T_0>0$ such that 
$ a_t  h(x)\Z^3 \not \in \mathcal  K$ for all  $t\ge T_0$.

In this section we will construct a fractal subset of $\mathrm{Sing}(w)$ whose
Hausdorff dimension is equal to that of $\mathrm{Sing}(w)$  using the above  dynamical interpretation and the idea of  shadowing.  
Roughly speaking
shadowing  means the following:
given $t_0\in \R $, if $x,y\in \R^2$ are close to each other (depending on $t_0$), then
$a_{t_0+t}h(x)\Z^3$ and $
a_{t_0+t}h(y)\Z^3
$ are close to each other for 
$t\in \R$ with $|t|\le C$ where $
C$ is a constant depending on $x$ and $y$.

The construction of the fractal structure starts with the lattice $\Z^3$. But all the results 
and proofs remain valid if  $\Z^3$ is replaced by a lattice in $\mathcal L_3'$, the subset  of $ \mathcal L_3$  defined in (\ref{eq;l prime}). This observation 
allows us to give a proof of Theorem \ref{thm;slice} at the end of this section.

\subsection{Construction of the  fractal set}\label{sec;construction}

We  define a fractal structure $(\mathcal T', \beta)$ on $\R^2$ inductively  for any choice of 
sequences of positive real numbers $\{\varepsilon_n\}_{n\in \N}$ and  $\{t_n\}_{n\in \N}$
with the following properties:
\begin{align}
& \varepsilon_n\le \varepsilon_{n-1} \mbox{ for all } n\in \N
	\mbox{ and } \varepsilon_n\to 0 \mbox{ as } n\to \infty, \label{eq;epsilon}\\
& t_{n}\ge t_{n-1}+1 \mbox{ for all } n\in \N \mbox{ and } t_{n+1}-t_n\to\infty\mbox{ as } n\to \infty,
\label{eq;t}
\end{align}
where  we set $t_0=0$ and $\varepsilon_0=1$ for convenience.

For $x=(x_1, x_2)\in \R^2$ and $r_1, r_2>0$ we let 
\[
I(x; r_1, r_2)=[x_1-r_1, x_1+r_1]\times [x_2-r_2, x_2+r_2]\subset \R^2. 
\]
We remark that  $\Z^3\in \mathcal L_3'$ and elements of 
$\mathcal L_3'$ will play the role of $\Z^3$ in our inductive construction 
of the fractal structure. 

The tree $\mathcal T'$ will have vertices in the set of  rational vectors $\Q^2$. 
We take the root of $\mathcal T'$ to be  $\tau_0=(0, 0)$ and define
\begin{align*}
\beta(\tau_0)&=I(\tau_0; \varepsilon_0 e^{-w_1t_1}, \varepsilon_0 e^{-w_2t_1}).
\end{align*}
Suppose we have defined the tree structure and  the map $\beta$  till height   $(n-1)$   of $\mathcal T'$.
For each vertex
$\tau_{n-1}\in \mathcal T'_{n-1}$ we want to define the set $\mathcal T'(\tau_{n-1})$ and  the map $\beta$  on 
it. This will complete the construction of the fractal structure.  
We define 
\begin{align}\label{eq;son}
\mathcal T'(\tau_{n-1})=\{\tau\in \beta(\tau_{n-1}): a_{t_{n}} h(\tau) \Z^3 \in \mathcal L_3' \}.
\end{align}
It is clear from the definition of $\mathcal L_3'$ and the assumption $t_{n}\ge t_{n-1}+1$   that
$\mathcal T'(\tau_{n-1})$ has empty intersection with $\bigcup_{0\le i\le n-1}\mathcal T'_i$.
For $\tau\in \mathcal T'(\tau_{n-1})$ we define 
\begin{align}\label{eq;box}
\beta(\tau)&=I(\tau; \varepsilon_{n} e^{-w_1t_{n+1}-t_{n}}, \varepsilon_{n} e^{-w_2 t_{n+1}-t_{n}}).\text{\footnotemark} 
\end{align}
\footnotetext{Recall that $\beta$ is a map from  $\mathcal T'$ which is identified with the vertices of the tree to compact subsets of $\R^2$. Each vertex $\tau \in \mathcal T'$ has a height $n$ and our definition of $\beta(\tau)$ also depends on $n$. Similar concerns apply in the definition of $\widetilde \beta $ below.}

It follows from (\ref{eq;son}) that for every $\tau \in \mathcal T_n\ (n\ge 0)$ there 
is a unique vector 
\begin{align}\label{eq;tau v}
\bv(\tau)\in \{r\be_3: 1/2<r\le 1\}\cap a_{t_{n}}h(\tau)\Z^3 .
\end{align} 
This property will be  used several times below. 
We end up this subsection by proving the following lemma. 
\begin{lem}\label{lem;contained}
	$	\mathcal F (\mathcal T', \beta)\subset \mathrm{Sing}(w)$.
\end{lem}

\begin{proof}
	
	Let $n\in \N\cup \{0\}$, $\tau\in \mathcal T'_n$ and 
	$x\in \beta(\tau)$. 
    Let $\bv(\tau)\in a_{t_n}h(\tau)\Z^3$ be as in (\ref{eq;tau v}).  
	Then for $t\in \R$ the lattice
	\[
	 a_t h(x)\Z^3 = a_t h(x-\tau)a^{-1}_{t_{n}} \cdot a_{t_{n}}h(\tau)\Z^3 
	\]
	contains the  primitive vector 
	\[
	 a_t h(x-\tau)a_{t_{n}}^{-1}\bv(\tau)
	\]
	whose norm is less than or equal to 
	\[
	3\max\{\varepsilon_{n} e^{-w_1(t_{n+1}-t)}, \varepsilon_{n} e^{-w_2( t_{n+1}-t)} ,e^{-(t-t_{n})}\}.
	\]

	Recall that we assume $w_1\ge w_2$. So for $n\in \N$ we solve the 
	 equation 
	\[
	\varepsilon_{n-1} e^{w_1(t-t_{n})}=e^{-(t-t_{n})}
	\]
	to get a unique solution $t=l_{n} $ with
	\[
	l_{n}-t_{n}=\frac{\log \varepsilon_{n-1}^{-1}}{1+w_1}\ge 0.
	\]
	Since $\varepsilon_n\to 0$ one has  $l_{n}-t_{n}\to \infty$ as $n\to \infty$. 
	
	Suppose $x\in \bigcap_{n\in \N\cup \{0\}} \beta(\tau_n)$ where $\{\tau_n\} 
	 \in \partial \mathcal T'$. For $n\in \N$ and  $t\in [t_{n}, l_{n}]$, the lattice 
	$ a_t h(x)\Z^3 $ contains the  primitive vector
	\[
	 a_t h(x-\tau_{n-1})a_{t_{n-1}}^{-1}\bv(\tau_{n-1})
	\]
	whose norm is less than or equal to 
	\begin{align}\label{eq;number 1}
	3\max \,\{ e^{-(l_n-t_n)}, e^{-(t_n-t_{n-1}}) \}.
	\end{align}
	For $t\in [l_{n}, t_{n+1}]$,  the lattice 
	$ a_t h(x)\Z^3 $ contains the primitive vector
	\[
	 a_t h(x-\tau_{n})a^{-1}_{t_{n}}\bv(\tau_{n})
	\]
	whose norm is less than or equal to 
	\begin{align}\label{eq;number 2}
	3 \max\{\varepsilon_n, e^{-(l_{n}-t_{n})}\}.
	\end{align}
	As the numbers in (\ref{eq;number 1}) and (\ref{eq;number 2}) tend to  zero as $n\to \infty$, 
	Mahler's compactness criterion (see \cite[Chapter \Rmnum{5}]{cassels}) implies  $x\in \mathrm{Sing}(w)$.

\end{proof}

\subsection{Refinement of the fractal structure}\label{sec;refinement}

We make explicit choices of the sequences
$\{\varepsilon_n\}, \{t_n\}$ and refine the tree  $\mathcal T'$ associated to them to get a
subtree $\mathcal T$ so that  $(\mathcal T, \beta)$ 
is a regular self-affine  structure satisfying the assumptions of Corollary \ref{cor;real real}.
In this subsection  we assume in addition that 
 $w_1>w_2$, although    our method also works in unweighted case where we 
 use first two conditions
 of (\ref{eq;key}) below  to define the subtree structure.
 In the lower bound estimate  we  will not go into details of unweighted case where   the Hausdorff dimension of $\mathrm{Sing}(w)$ is known.

Let $\til c, \til \varepsilon, \til s\le 1$ be 
positive real numbers as in Lemmas \ref{lem;many vectors},
 \ref{lem;application 1} and 
 \ref{lem;application 2}  respectively. 
We fix $\varepsilon,  t,r>0$ with the following properties 
\begin{itemize}
	\item [(\rmnum{1})] $0< \varepsilon< r <\frac{1}{10^4}\min\{ \til \varepsilon, \til s, \til c,w_2,  w_1-w_2\}$;
	\item [(\rmnum{2})]  $ t=\frac{100}{\varepsilon^2 }$.
\end{itemize}
The sequence $\{\varepsilon_n\}$ and $\{t_n\}$ are defined by  
\begin{itemize}
	\item [(\rmnum{3})]  $\varepsilon_n=\frac{\varepsilon}{n}$ for  $n\in \N$;
	\item [(\rmnum{4})]  $t_{n}-t_{n-1}= n t$ for  $n\in \N$. \footnote{Recall that 
	$t_0=0$ and $\varepsilon _0=1$.}
\end{itemize}
It is not hard to see that  for  any integer  $n\ge 0$ one has
\begin{align}
\label{eq;why}
\varepsilon_n^{-100}\le \min\{ e^{w_2nt}, e^{(w_1-w_2)nt}\}. 
\end{align}

It can be checked directly that  (\ref{eq;epsilon}) and (\ref{eq;t})  hold for the sequences
$\{\varepsilon_n  \}_{n\in \N}$ and $\{t_n \}_{n\in \N}$. Hence they define a
fractal structure $(\mathcal T', \beta)$ with $\mathcal F(\mathcal T', \beta)\subset \mathrm{Sing}(w)$
by Lemma \ref{lem;contained}. 
For $n\in \N\cup \{ 0 \}$ we let 
\begin{align*}
b_n&=\mathrm{diag}(e^{-w_2nt},e^{w_2nt}, 1 )\in SL_3(\R)   \\
\til \beta(\tau)& =I(\tau; \varepsilon_{n+1} e^{-w_1t_{n+1}-t_{n}}, \varepsilon_{n+1} e^{-w_2 t_{n+1}-t_{n}})
\quad \mbox{where}\quad  \tau \in \mathcal T_{n}. 
\end{align*}
From the definition, it is evident that $\til \beta(\tau)\subset \beta(\tau)$ for $\tau \in \mathcal T_{n}$.

Let $\mathcal T$ be the rooted subtree of $\mathcal T'$ defined in the following way: 
$\tau_0=(0,0)$ is the root of $\mathcal T$ and 
$\mathcal T(\kappa)$ ($\kappa \in \mathcal T_{n-1}$) consists of all the 
$\tau \in \til \beta(\kappa)$ with the following properties:
\begin{equation}
\label{eq;key}
\begin{aligned}
&a_{t_n}h(\tau )\Z^3\in \mathcal L_3', \\
&a_{t_n}h(\tau)\Z^3  \in \mathcal K_{\varepsilon_n^2}^*, \\
&b_n a_{t_n}h(\tau)\Z^3 \in \mathcal K_{r}^*,
\end{aligned}
\end{equation}
where $\mathcal K_r^*$ is defined in (\ref{eq;mel}).
It can be checked directly  that   $\beta(\tau) \subset \beta(\kappa)$ for all $\tau \in 
\mathcal T(\kappa)$ (this  is the main reason for  using $\til \beta$).
It will follow from  Lemma \ref{lem;plenty} below that  
 each  vertex of $\mathcal T$ has nonempty set of sons. 
 Therefore  $(\mathcal T, \beta)$ is a regular self-affine structure.
\begin{lem}\label{lem;plenty}
	For every  $n\in \N$ and  $y\in \mathcal T_{n-1}$ one has   
	\[
\frac{1}{100} \varepsilon_{n}^2 e^{2 n t }\le 	\sharp \mathcal T(y)\le 10 \varepsilon_{n}^2 e^{2 n t }.
	\]
\end{lem}

Let us fix   $n\in \N$,  $y\in \mathcal T_{n-1}$.
We first reduce the calculation of $\sharp \mathcal T(y)$ to lattice points counting  in $\R^3$ so that 
we can use the  results of \S \ref{sec;count r3}. 
We set 
\begin{equation}
\label{eq;lambda 1}
\begin{aligned}
\Lambda&=a_{t_{n-1}}h(y)\Z^3\in \mathcal L_3'\cap \mathcal K_{\varepsilon_{n-1}^2}^*\subset
\mathcal L_3'\cap \mathcal K_{\varepsilon_{n}^2}^*, \\
\Lambda_1&=a_{t_{n}}h(y)\Z^3=a_{nt}\Lambda,   \\
 \Lambda_2&=b_na_{t_{n}}h(y)\Z^3 =b_n a_{nt} \Lambda.
\end{aligned}
\end{equation}
Given $x\in  \widetilde \beta(y)$, to have  $x\in \mathcal T(y)$ 
the lattices
\begin{equation}\label{eq;lambda x}
\begin{aligned}
\Lambda_1(x)& =a_{t_n}h(x)\Z^3 =a_{t_n}h(x-y)a^{-1}_{t_{n}} \Lambda_1 \quad \mbox{and}\\
\Lambda_2(x)& =b_na_{t_n}h(x)\Z^3 =b_na_{t_n}h(x-y)a^{-1}_{t_{n}} b_n^{-1}\Lambda_2,
\end{aligned}
\end{equation}
must  satisfy   $ \Lambda_1(x)\in\mathcal  K^*_{\varepsilon_{n}^2}, \Lambda_2(x)\in \mathcal K^*_{r}$ and $\Lambda_1(x)\in \mathcal L_3'$
(which implies $\Lambda_2(x)\in \mathcal L_3'$).
Therefore Lemma \ref{lem;plenty} follows from the following  lemma.
\begin{lem}\label{lem;card}
	Let $n\in \N$ and $y\in \mathcal T_{n-1}$. Then  
	\begin{align}
	 \frac{1}{10} \varepsilon_{n}^2 e^{2n t }\le \sharp \{ x\in  \til\beta(y): \Lambda_1(x) \in \mathcal L_3'  \}&\le 10\varepsilon_{n}^2 e^{2n t };
	 \label{eq;card} \\
	\sharp \{ x\in  \til \beta(y)  :\Lambda_1(x) \in \mathcal L_3' \setminus \mathcal K_{\varepsilon_n^2}^*                         \}
	&\le \frac{8}{100} \varepsilon_{n}^2 e^{2nt }; 
	\label{eq;card 1}\\
	\sharp \{ x\in \til\beta(y)  : \Lambda_2(x) \in \mathcal L_3' \setminus \mathcal K_{r}^*                         \}
	&\le \frac{1}{100} \varepsilon_{n}^2 e^{2 nt }.
	\label{eq;card 2}
	\end{align}
\end{lem}
\begin{proof}
	We first prove (\ref{eq;card}).
	Suppose $\Lambda_1(x)\in \mathcal L_3'$ where $x\in \til \beta(y)$. Then		
	there exists $s_x$ with $1/2<s_x\le 1$ such that $\Lambda_1(x)$ contains a  primitive vector 
	$\bv(x)=s_x \be_3$.
		It follows from the definition of  $\til  \beta(y)$ and a direct calculation that  the vector  $a_{t_n} h(y-x)a_{t_n}^{-1}s_x\be_3$ 
		belongs to 
		\[
	M=\{(z_1, z_2, z_3):	\max\{|z_1|,|z_2|\}\le \varepsilon_{n} e^{nt}|z_3| \quad\mbox{and} \quad 1/2<|z_3|\le 1\}.
		\]
		It is not hard to see that the map $x\to a_{t_n} h(y-x)a_{t_n}^{-1}\bv(x)$ is a bijection between the sets 
		 $\{x\in \til \beta(y): \Lambda_1(x)\in \mathcal L_3' \}$ and $M \cap \widehat \Lambda_1$. 	
    Let 
	\begin{align*}
	M^{(1)}& =\{(z_1, z_2, z_3)\in \R^3 :\max\{ |z_1|,|z_2|\}\le \frac{1}{2} \varepsilon_{n} e^{nt  },  |z_3|\le 1\}\\
	M^{(2)}& =\{(z_1, z_2, z_3)\in \R^3 :\max\{ |z_1|,|z_2|\}\le \frac{1}{2} \varepsilon_{n} e^{nt  },  |z_3|\le \frac{1}{2}\}.
	\end{align*}
	Then $(M^{(1)}\setminus M^{(2)})\subset M \subset  2 	M^{(2)}$. 	
	It follows that 
	\begin{align}\label{eq;card so}
	\sharp M^{(1)}\cap \widehat \Lambda_1-\sharp M^{(2)}\cap \widehat \Lambda_1\le 
	\sharp \big( M\cap \widehat \Lambda_1\big)\le 
	\sharp (2 	M^{(2)})\cap \widehat \Lambda_1 .
	\end{align}

	Using (\ref{eq;minkowski}) and  $\Lambda\in \mathcal K_{\varepsilon_{n}^2}^*$ one has
	\[
	\lambda_1(B_1, \Lambda)\ge 100^{-1} \varepsilon^4_{n}\quad \mbox{and} \quad \lambda_3(B_1, \Lambda)\le 100 \varepsilon^{-2}_{n}
	\]	
	where $B_r=B_r(\R^3)$ in this section. 
	Recall that  $a_{nt}^{-1}\Lambda_1=\Lambda$ by (\ref{eq;lambda 1}) .
	For $i=1, 2$ we have 
	\begin{gather*}
	\lambda_1(M^{(i)},  \Lambda_1)=\lambda_1(a_{nt}^{-1}M^{(i)} , \Lambda)\\
	\ge \lambda_1(a^{-1}_{nt}M^{(1)} , \Lambda ) 
	\ge \lambda_1(B_{3e^{nt}},\Lambda)
	\ge (300)^{-1} e^{-nt}\varepsilon^4_{n}
	\end{gather*} 	
	and 
	\begin{gather}
	\lambda_3(M^{(i)},  \Lambda_1)=\lambda_3(a^{-1}_{nt}M^{(i)} , \Lambda) \label{eq;lambda 3}\\
	\le \lambda_3(a_{nt}^{-1}M^{(2)},\Lambda)
	\le \lambda_3(B_{\frac{1}{2}\varepsilon_{n} e^{w_2nt}},\Lambda)
	\le 200  e^{-w_2nt}\varepsilon^{-3}_{n}.\notag
	\end{gather} 	
	By these estimates and (\ref{eq;why})
	it is straightforward  to check that the assumptions of 
	Lemma \ref{lem;many vectors} for $M^{(1)}$,  $M^{(2)}, 2M^{(2)}$ and $\Lambda$ are satisfied. 
	Therefore  (\ref{eq;card so}) and Lemma \ref{lem;many vectors} imply 
	\[
 (5\zeta(3))^{-1} \varepsilon_n^2 e^{2nt}\le 	\sharp \big( M\cap \widehat \Lambda_1\big)\le 48 (5\zeta(3))^{-1} \varepsilon_n^2 e^{2nt}. 
	\]
	To complete  the proof of  (\ref{eq;card}), it suffices to note that  $1< \zeta(3)< 2$.

Next we prove (\ref{eq;card 1}) and (\ref{eq;card 2}) together. 
	Let $s_1=\varepsilon_n^2,s_2=r, a^{(1)}=a_{nt}$,  $a^{(2)}=b_na_{nt}$ and   for $i=1, 2$
	\[
	\mathcal S_i=\{x\in \til \beta(y): \Lambda_i(x)
	 \in\mathcal L_3'\setminus \mathcal K_{s_i}^*\}.
	\]
    We will show  that 
	\begin{align}\label{eq;kuaidi}
	\sharp \mathcal S_1 \le 8\sqrt {\varepsilon_n} \varepsilon_n^2e^{2nt}\quad \mbox{and}\quad 
	\sharp \mathcal S_2 \le 8 r \varepsilon_n^2e^{2nt}.
	\end{align}
    In view of the definitions of $\varepsilon_n$ and $r$, this will complete the proof.

	Let $\Lambda_i$ and $\Lambda_i(x)$  be as in (\ref{eq;lambda 1}) and (\ref{eq;lambda x}) respectively.  Let $\bv(x)\in \Lambda_1(x)\cap \Lambda_2(x)$ be as in (\ref{eq;tau v}).   Let  
	$x^{(i)}\in \R^2 $  be  such that $h(x^{(i)})=a^{(i)}a_{t_{n-1}}h(y-x)(a^{(i)}a_{t_{n-1}})^{-1}$. Then 
	\begin{align*}
	\bw_i(x): = h(x^{(i)})\bv(x)\in \Lambda_i.
	\end{align*}
    It can be calculated that  for all $x\in \widetilde{\beta}(y)$
	\begin{equation}\label{eq;x i j}
	\begin{aligned}
    &\max \{|x^{(1)}_1|, |x^{(1)}_2|\}\le \varepsilon_{n}e^{nt },\\
     & |x^{(2)}_1|\le  \varepsilon_{n}e^{w_1nt },\quad  |x^{(2)}_2|
	\le  \varepsilon_{n}e^{(w_1+2w_2)nt }.
	\end{aligned}
	\end{equation}
	Let $M_i=M_{\br_i}$ (see (\ref{eq;m r}) for the definition) where 
	\begin{align*}
	\br_1&=(\varepsilon_{n}e^{nt },\varepsilon_{n}e^{nt },1 )=(r_{11}, r_{12}, r_{13})
	\\ \br_2&=(\varepsilon_{n}e^{w_1nt },\varepsilon_{n}e^{(w_1+2w_2)nt },1)=(r_{21}, r_{22}, r_{23}).
	\end{align*}
	The map $\mathcal S_i\to M_i\cap \widehat \Lambda_i$ with $x\to \bw_i(x)$ is injective. 
	If   for all $x\in \mathcal S_i$ there exists 
	$\varphi_i\in N_{3s_ir_{i2}}(\br_i, s_i)\cap \widehat {\Lambda_i^*}$ (see (\ref{eq;jia}) for the definition 
	of $N_q(\br, s)$) such that $\varphi_i(\bw_i(x))=0$,
	then 
	\[
	\sharp \mathcal S_i\le \sharp \mathcal S(\Lambda_i, \br_i, s_i)=\sharp \mathcal S(a^{(i)}\Lambda, \br_i, s_i)
	\]
	where $\mathcal S(\Lambda, \br, s)$ is defined in (\ref{eq;jia1}). 
	Therefore  the two estimates of (\ref{eq;kuaidi}) will   follow from 
	  Lemmas \ref{lem;application 1}  and \ref{lem;application 2} respectively. 
	Here the assumptions of these two lemmas  can be checked easily using (\ref{eq;lambda 1}) and  the assumptions (\rmnum{1})-(\rmnum{4})
	at the beginning of this section.

    Suppose $x\in \mathcal S_i$. We prove that $\varphi_i(\bw_i(x))=0$ for some $\varphi_i\in 
    N_{3s_ir_{i2}}(\br_i, s_i)\cap \widehat {\Lambda_i^*}$.
     It follows from the definition of $\mathcal S_i$ that   
	$a^{(i)}a_{t_{n-1}}h(x)\Z^3
	\not \in \mathcal K_{s_i}^*$. 
	So  there exists   $\varphi_{i}\in \widehat{\Lambda_i^*}$ such that  $\|h(x^{(i)})^*\varphi_i\|<s_i$
	where $h(x^{(i)})^*$ is the adjoint action defined by $g^*\varphi(\bv)=\varphi(g\bv)$ for all $g\in SL_3(\R)$ and $\bv\in \R^3$.
	We claim that $\varphi_i(\bw_i(x) )=0$. Note that $\bw_i(x)\in \Lambda_i$ and $\varphi_i\in\widehat{\Lambda_i}$ implies $\varphi_{i}(\bw_i(x) )\in \Z$. Then the claim follows from 
 \begin{align*}
 |\varphi_{i}(\bw_i(x) )|&=|h(x^{(i)})^*\varphi_{i} (h(-x^{(i)})\bw_i(x) )|=|h(x^{(i)})^*\varphi_{i}(\bv(x))| \\ &   \le | h(x^{(i)})^*\varphi_{i}(\be_3)|\le 
 \|h(x^{(i)})^*\varphi_{i}\|<s_i<1.
 \end{align*}
	From   direct calculations we have 
	\begin{align*}
	h(x^{(i)})^*\varphi_i=(\varphi_i(\be_1),\varphi_i(\be_2) ,x^{(i)}_1\varphi_i(\be_1)+x^{(i)}_2\varphi_i(\be_2)+ \varphi_i(\be_3)) .
	\end{align*}
	It follows from (\ref{eq;x i j}) and the fact  $\|h(x^{(i)})^*\varphi_{i}\|<s_i$ that 
	\begin{align*}
	\max\,\{|\varphi_i(\be_1)|, |\varphi_i(\be_2)|\}< s_i \quad \mbox{and }\quad |\varphi_i(\be_3)|< 3 s_i r_{i2}.
	\end{align*}
	Therefore $\varphi_i\in N_{3s_ir_{i2}}(\br_i, s_i)$ by (\ref{eq;jia}) and this completes the proof. 
\end{proof}

\subsection{The lower bound calculation}

In this subsection we complete the proof of the lower bound. 
\begin{thm}\label{thm;lower bound}
	Let $w=(w_1, w_2)$ where $w_1> w_2>0$ and $w_1+w_2=1$. 
	Then $\dim_{H} \mathrm{Sing}(w)\ge 2-\frac{1}{1+w_1}  $.
\end{thm}


\begin{prop}\label{prop;transfer}
Suppose $w_1>w_2>0$ and let  $(\mathcal T, \beta)$ be the self-affine structure on $\R^2$ defined in the previous section. 
Then \[
\mathrm{dim}_H\mathcal F(\mathcal T, \beta)\ge 2-\frac{1}{1+w_1}. 
\]
\end{prop}
Our tool is Corollary \ref{cor;real real}.  Let 
$t $ and $ \varepsilon$ be  constants fixed at the beginning  of  \S \ref{sec;refinement}. 
It is clear from its constructions that  $(\mathcal T, \beta)$
 is a regular self-affine structure satisfying assumptions (1) and (2) of 
 Theorem \ref{thm:lower_bound} with 
\begin{align*}
 C_n = \frac{\varepsilon^2}{100n^2} e^{2n t}, \quad 
 W_n= \frac{2\varepsilon }{n} e^{- w_1t_{n+1}-t_n}\quad \mbox{and}\quad 
 L_n  =\frac{2\varepsilon }{n} e^{- w_2t_{n+1}-t_n}
\end{align*}
where $n\ge 1$ and 
  $t_n=\sum_{i=0}^n it=  n(1+n)t/2$. \footnote{For $n=0$ we take  $W_0=e^{-w_1t}$ and $L_n=e^{-w_2t}$ and $C_0=1$.}
We will see from the following lemma about 
well-separated property of the fractal structure that assumption (3) of  Theorem \ref{thm:lower_bound} holds for 
\begin{align*}
 \rho_n=e^{-w_1n t }
\end{align*} 
provided $n$ is sufficiently  large.

\begin{lem}\label{lem;separation}
	Let $\tau \in \mathcal T_{n-1} \  (n\in \N)$. Then for all different  $x, y\in \mathcal T(\tau)$
	one has 
	\[
	\mathrm{dist}\,(\beta(x), \beta(y))\ge W_{n-1} \cdot \frac{r}{8\varepsilon_{n-1}}\min\{
	e^{-w_1nt},  e^{(w_1-w_2)t_{n}-(1+w_2)nt}\}.
	\]
\end{lem}
\begin{rem}
It will be clear from the proof that for $n$ sufficiently large  either $\beta(x)$ and $\beta(y)$ 
have horizontal distance  at least 
$ e^{-w_1nt}W_{n-1}$ (which is $\gg W_n$) or they have vertical distance at least $L_{n-1}e^{-nt-w_2nt}$
(which is $\approx  L_n$).  If we do not  assume the third
condition of (\ref{eq;key}), the same argument below will give the corresponding  horizontal   (resp.~vertical) separation  $e^{-nt}\varepsilon_n^2W_n$  (resp. $e^{-nt}\varepsilon_n^2L_n$). But then the possible  horizontal separation is too small and the assumption (\rmnum{4}) of Corollary \ref{cor;real use}  no longer holds. 
The validity of Corollary \ref{cor;real use} (\rmnum{4}) means that, roughly speaking, nearby
$\beta(x)$ and $\beta(y)$ have either  horizontal distance $ e^{-w_1nt}W_{n-1}$ or  vertical distance
$L_{n-1}e^{-nt-w_2nt}$.
\end{rem}

\begin{proof}
	Since $x,y\in \mathcal T(\tau)$, 
	there are $1/2< s, l\le 1$ such that 
	\[
	s \be_3\in b_{n}a_{t_{n}}h(x)\Z^3\quad \mbox{and}\quad l \be_3\in b_{n} a_{t_{n}}h(y)\Z^3. 
	\]
	Let
	\[
	\bv=b_{n} a_{t_{n}} h(y-x)(b_{n} a_{t_{n}})^{-1} s\be_3\in  b_{n} a_{t_{n}}h(y)\Z^3. 
	\]
	Since $ b_{n} a_{t_{n}}h(y)\Z^3\in\mathcal K_r^*$, one has 
	\begin{align*}
	\|\bv\wedge l\be_3\| &=ls  \|((y_1-x_1)e^{t_n+w_1t_n -w_2nt},(y_2-x_2)e^{t_n+w_2t_n +w_2nt} )\|  
	 \ge r.
	\end{align*}
	Then either 
	\begin{enumerate}[label=(\roman*)]
		\item $|y_1-x_1|e^{t_n+w_1t_n -w_2nt}\ge r/2$ \qquad \mbox{or}
		\item $|y_2-x_2|e^{t_n+w_2t_n +w_2nt}\ge r/2$.
	\end{enumerate}
	Let  $x'\in \beta(x) $ and $y'\in \beta(y)$.  
	If (\rmnum{1}) holds  then
	\begin{align*}
	\|y'-x'\|&\ge |y'_1-x'_1| \\
	& \ge |y_1-x_1|-|x_1-x_1'|-|y_1-y_1'|\\
	&\ge e^{-t_n-w_1t_n+w_2nt}(r/2-2\varepsilon_n  e^{-nt})\\
	&\ge e^{-t_n-w_1t_n+w_2nt}r/4\\
	&=W_{n-1}\cdot \frac{r}{8\varepsilon_{n-1}}  e^{-w_1nt}.
	\end{align*}
	If (\rmnum{2}) holds, then
	\begin{align*}
	\|y'-x'\|&\ge |y'_2-x'_2| \\
	& \ge |y_2-x_2|-|x_2-x_2'|-|y_2-y_2'|\\
	& \ge e^{-t_n-w_2t_{n}-w_2nt}(r/2-2\varepsilon_n e^{-w_2t})\\
	& \ge e^{-t_n-w_2t_{n}-w_2nt}r/4\\
	& =  W_{n-1}\cdot \frac{r}{8\varepsilon_{n-1}} e^{(w_1-w_2)t_{n}-(1+w_2)nt}.
	\end{align*}
This completes the proof. 
\end{proof}

\begin{proof}[Proof of Proposition \ref{prop;transfer}]
	We will apply   Corollary \ref{cor;real real} which uses the local data.  We have 
	\begin{align*}
	 W_n/W_{n-1}=\frac{n-1}{n} e^{-(n+1) t w_1 -n t} \\
	 L_n/L_{n-1}=\frac{n-1}{n} e^{-(n+1) t w_2 -n t}
	\end{align*}
	for $n\ge 2$.
	It can be checked directly that for any integer  $k$ with $k\ge \frac{w_1}{w_2}+10$ the  assumptions 
    of Corollary \ref{cor;real real}  hold. Moreover, we have
	\[
    \lim_{n\to \infty}  \frac{\log ({L_n} C_n/{L_{n-1}} )}{-\log (W_n/W_{n-1})}=\frac{w_1}{1+w_1}. 
	\]
	Therefore Corollary \ref{cor;real real} implies $\dim_H \mathcal F(\mathcal T, \beta)\ge 1+\frac{w_1}{1+w_1}
	= 2-\frac{1}{1+w_1}$.
\end{proof}

\begin{proof}[Proof of Theorem \ref{thm;lower bound}]
	If $w>w_2$, then the conclusion follows from Proposition \ref{prop;transfer} and Lemma \ref{lem;contained}. If $w_1=w_2$ then the conclusion follows from \cite[Theorem 1.1]{c11}.
\end{proof}

\begin{proof}[Proof of Theorem \ref{thm;slice} (sketch)]
Note that the set 
\[
Q_\Lambda:=\{y\in \R^2:  { h(y ) \Lambda }\cap \R \be_3  \neq \{ 0\} \}
\]
is dense in $\R^2$. 
We fix $y\in U\cap Q_\Lambda$ and $s\in \R$ such that $a_s h(y)\Lambda \in\mathcal L_3'$. 
In our construction of the fractal structure $(\mathcal T,\beta)$ in \S \ref{sec;construction} and \S \ref{sec;refinement}
we only use the property $\Z^3\in \mathcal L_3'$. 
So the same construction 
will give a fractal structure $(\mathcal T'' , \beta'')$ such that the Hausdorff dimension of  $\mathcal F(\mathcal T'', \beta'')\subset I(0; e^{-w_1 t}, e^{-w_2 t})$ is at least $2-\frac{1}{1+w_1}$  and for any $x\in \mathcal F(\mathcal T'', \beta'')$	the trajectory 	
$\mathcal A^+ h(x)a_s h(y)\Lambda$ is divergent. 
By taking $t$ sufficiently large, we can make sure that 	
\[
a_{s}^{-1} h (\mathcal F(\mathcal T'', \beta'')) a_s h(y)\subset h(U). 
\]
This implies that $ \{ x\in U : \mathcal A^+h(x)\Lambda \mbox{ is  divergent }\}$ contains 
 $y+g \mathcal F(\mathcal T'', \beta'')  $ for some nonsingular linear transformation $g$ of 
 $\R^2$. 
Therefore the conclusion holds.
\end{proof}

\section{Best approximation and upper bound}\label{sec;upper best}
We first review the definition of $w$-weighted  best approximation and use it to construct a self-affine covering 
of 
\begin{align*}
{\mathrm{Sing}}(w) ^*&:=\{ x \in  {\mathrm{Sing}(w)}: 1, x_1, x_2 \mbox{ are linearly independent over }\Q\}
\end{align*} 
in \S \ref{sec;best}. By Khintchine's transference principle (\cite[Chapter IV, \S 5]{schmidt}), it is not hard to see that all $x \in \R^2$ with 
$1, x_1, x_2$ linearly dependent over $\Q $ are $w$-singular. Note that the set of these $x$ is a countable union of lines in $\R^2$. Thus the Hausdorff dimension of 
${\mathrm{Sing}}(w) \setminus {\mathrm{Sing}}(w) ^*$ is one. We prove in \S \ref{sec;upper} that the upper bound of the Hausdorff 
dimension of the fractal associated to  the self-affine structure  constructed in \S \ref{sec;best} can be arbitrarily close to 
$2-\frac{1}{1+w_1}> 1$. 
Therefore the Hausdorff dimension of $\mathrm{Sing}(w)$ is bounded from above by $2-\frac{1}{1+w_1}$.

\subsection{Best approximation and self-affine covering}\label{sec;best}

We define $w$-weighted quasi-norm
on $\mathbb R^2$ by 
\[
\|(x_1, x_2)\|_w=\max \{|x_1|^{1/w_1}, |x_2|^{1/w_2}\}.
\]
Although it does not satisfy  the triangle inequality, using convexity of the function $s\to s^{1/w_i}$ we have
\begin{align}\label{eq;weighted triangle}
\| x + y  \|_w\le 2^{w_1/w_2}(\|x\|_w+\|y\|_w) \quad \mbox{for all } x,y \in \R^2.
\end{align}
We say $({ p }, {q})\in \Z^2\times \N$ 
is a best approximation vector  of $ x \in \R^2$ with respect to  the  
$w$-weighted norm
if 
\begin{enumerate}[label=(\roman*)]
	\item $\|q x - p \|_w<\|q' x -p'\|_w$ for any $(p', q')\in \mathbb Z^2\times \N$ with  
	$q'<q$;
	\item $\|q x  - p \|_w\le \|q x - p '\|_w$ for any $ p '\in \mathbb Z^2$. 
\end{enumerate}
For simplicity we call $(p, q)$ a $w$-best approximate of $x$.

There is a naturally defined bijection between $\Q^2$ and 
\[
Q=\{( p , q)=(p_1, p_2, q)\in \Z^2\times \N: \mathrm{gcd}( p_1, p_2 , q)=1\}\subset \widehat  {\Z^3}, 
\]
namely, every $\bu=(p ,q)\in Q$ corresponds to  $\widehat \bu=\frac{p}{q}$.
For  $x\in \R^2$ and  $\bu=(p ,q)\in \Z^2\times \Z$ we let 
\[
|\bu|=q \quad \mbox{and}\quad  A( x , \bu )=\|q x - p \|_w  . 
\]
Let 
\begin{align}\label{eq; r v}
r(\bu )=\min_{\bv\in Q, \bv\neq \bu }A(\widehat \bu , \bv)
\end{align}
which is the  best approximation of $\widehat \bu$   
by $ \bv\in Q\setminus \{\bu\}$. 
We will see in  (\ref{eq;rv min}) that $r(\bu)$ is the  smallest $w$-weighted norm of nonzero vectors of 
a lattice in $\R^2$.  
For every $x\in \R^2\setminus \Q^2$ we  associate a sequence $\Sigma_ x =\{\bu _i\}_{i\in \N}\subset Q$
of  $w$-best approximates  of $x$ with the following  properties:
\begin{itemize}
	\item $|\bu_1|>1$; 
	\item 	$|\bu_i|< |\bu_{i+1}|$ for all  $i\in \N$ ;
	\item there is no  $w$-best approximate  $(p, q)$ of $x$ with 
	$|\bu_i|<q < |\bu_{i+1}|$.
\end{itemize}

\begin{lem}\label{lem;balance}
	Let  $\bu \in Q$ with $|\bu|>1$. Then for any $\bv\in Q$ with $|\bv|<|\bu|$ one has  
	$A(\widehat \bu , \bv)=A(\widehat \bu , \bu -\bv)$.
\end{lem}
\begin{proof}
	Suppose $\bu=(p,  q)\in \Z^2 \times \N$ and $\bv=(s, l)\in \Z^2\times \N$. 
	Then $\bu-\bv=(p-s, q-l)$ and
	 $$A(\widehat \bu , \bu -\bv)=\|(q-l) {p}{q^{-1}}-(p-s)\|_w=\|-(l{p}{q}^{-1}-s)\|_w=A(\widehat \bu, \bv).$$	
\end{proof}

\begin{cor}\label{lem;half weight}
	Let  $\bu \in Q$ with $|\bu|>1$. Then there exists $\bv\in Q$ with $|\bv|\le \frac{|\bu |}{2}$ such that 
	$r(\bu )=A(\widehat \bu , \bv)$.
\end{cor}
\begin{proof}
	Since $|\bu|>1$ one has $r(\bu)=\min_{\bv\in Q, |\bv| <|\bu| }A(\widehat \bu , \bv)$.
	So the corollary follows from Lemma \ref{lem;balance}.
\end{proof}

\begin{lem}\label{lem;estimate}
	Let $\bu \in Q$ with $|\bu |>1$   be a $w$-best approximate of $ x  \in \R^2$.
	Then 
	$
	A( x , \bv)< 2^{1/w_2} A(\widehat \bu , \bv)
	$
	for all $\bv\in Q$ with   $|\bv|\le  |\bu |/2$. 
\end{lem}
\begin{proof}
	It follows from the definition of best approximate that 
	\begin{align}\label{eq;apply best}
	A( x , \bu )& < A( x ,  \bv) .
	\end{align}
	Let $\bu =(p_1, p_2, q), \bv=(s_1 ,s_2, l)$ and choose $i\in \{1, 2\}$ such that 
	$A( x , \bv)=|lx_i-s_i|^{1/w_i}$. The choice of $i$  implies
	
	\begin{align}
	&& A( x , \bv)^{w_i}&= l|x_i-s_i/l|\notag \\
	&&           & \le l  \left(|x_i-p_i/q|+|p_i/q-s_i/l|\right)\notag \\
	&&             & = \frac{l}{q}  |qx_i-p_i|+|lp_i/q-s_i|\notag \\
	&&             & \le \frac{1}{2} A( x , \bu )^{w_i}+A(\widehat \bu , \bv)^{w_i}\notag   \\
	&&           & < \frac{1}{2} A( x , \bv )^{w_i}+A(\widehat \bu , \bv)^{w_i}. &&(\mbox{by (\ref{eq;apply best})}) \notag
	\end{align}
	Therefore
	\[
	A( x , \bv)< 2^{1/w_2}A(\widehat \bu , \bv).
	\]
	
\end{proof}

Let  $\bu \in Q$ and 
$
B_w(\bu, r)=\{x\in \R^2: A(x, \bu)< r \}. 
$
It is not hard to see that $B_w(\bu, r)$ is an open rectangle with 
center $\widehat{\bu}$. 
The 
set of $x\in \R^2$ which has $\bu\in Q$ as a $w$-best approximate  is 
\begin{align}\label{eq;domain}
\Delta (\bu )= \left (\bigcap_{|\bv|<|\bu|  }\Delta_{\bv}(\bu )\right ) \cap
 \left(\bigcap_{|\bv|=|\bu |,  \bu\neq \bv}\overline {\Delta_{\bv}(\bu )} \right).
\end{align}
where 
\begin{align*}
\Delta_{\bv}(\bu )=\{ x \in \R^2: A( x , \bv)>A( x , \bu )\}.
\end{align*}
The following lemma says that  $\Delta (\bu )$ is roughly the rectangle $B_w(\bu, r(\bu))$. 
\begin{lem}\label{lem;size delta}
	For any $\bu \in Q$ with  $|\bu |>1$ one has 
	\[
	B_w({\bu , {2^{-1/w_2}}r(\bu)})\subset \Delta (\bu )\subset B_w(\bu , 2^{1/w_2}r(\bu))).
	\]
\end{lem}
\begin{proof}
	We write   $r=r(\bu )$ to simplify the notation. 
	Let $\bu =(p_1, p_2, q)$ and suppose $ x =(x_1, x_2)\in B_w(\bu ,2^{-1/w_2}r)$. 
	Let $\bv=(s_1, s_2, l) \in Q$ with $|\bv|\le |\bu |$ and $\bv\neq \bu$. 
	It follows from the  definitions of $B_w(\bu ,2^{-1/w_2}r)$ and $r(\bu )$ that 
	\begin{align}
	A( x , \bu ) < 2^{-1/w_2}r
	\le 2^{-1/w_2}A(\widehat \bu , \bv).\label{eq;half}
	\end{align}
	We choose $i\in \{1,2\}$ such that 
	\begin{align}\label{eq;choice 1}
	A(\widehat \bu , \bv)=|lp_i/q-s_i|^{1/w_i}.
	\end{align}
	Then
	\begin{equation}\label{eq;yufei}
	\begin{aligned}
	 A( x , \bv)^{w_i}& \ge |lx_i-s_i|  \\
	 &=l|x_i-{s_i}/{l}| \\
	 & \ge l (|p_i/q-s_i/l|- |x_i-p_i/q|) \\
	 & = |lp_i/q-s_i|-\frac{l}{q}|qx_i-p_i|\\
	   &\ge A(\widehat \bu , \bv)^{w_i}-A( x , \bu )^{w_i},
	\end{aligned}
	\end{equation}
    where in the last  inequality we use  	 (\ref{eq;choice 1}).
    In view of 
  (\ref{eq;yufei}) and   (\ref{eq;half}) one has 
	$
	A( x , \bv)>A( x , \bu )
	$,
   from which one has  $ x \in \Delta_{\bv}(\bu )$. 
	Since  $\bv $ is an arbitrary element of $Q\setminus \{\bu \}$ with $|\bv|\le |\bu|$, 
	the definition of $\Delta (\bu)$ in   (\ref{eq;domain}) implies 
	$ x \in \Delta (\bu )$. 
	
	Now suppose $ x \in \Delta (\bu )$.  Corollary \ref{lem;half weight} implies that  there exists 
	$\bv$ with
	$|\bv|\le  |\bu |/2$
	such that  $r=A(\widehat \bu , \bv)$.  
	It follows from 
	Lemma \ref{lem;estimate} that 
	\[
	A( x , \bu )< A( x , \bv)< 2^{1/w_2}r.
	\]
	
\end{proof}

\begin{lem}\label{lem;bound sequence}
	Let $x\in \R^2\backslash \Q^2$ and  $\Sigma_x=\{\bu _i\}_{i\in \N}$.  Then
	for  all $i,j\in \N$ one has
	\begin{align}\label{eq;best bound}
	{2^{-1/w_2}} A(\widehat \bu _{i+j}, \bu _i)<A( x , \bu _i)< 2^{1/w_2}{r(\bu _{i+1})}.
	\end{align}
\end{lem}

\begin{proof}
	By Corollary \ref{lem;half weight} there exists   
	$\bv\in Q$ with
	$|\bv|\le  |\bu _{i+1}|/2$
	such that  $r(\bu _{i+1})=A(\widehat \bu _{i+1}, \bv)$.  
	It follows from 
	Lemma \ref{lem;estimate} that 
	\begin{align}\label{eq;asymptotic1}
	A( x , \bv)< 2^{1/w_2}r(\bu _{i+1}).
	\end{align}
	On the other hand the definition of $w$-best approximate implies 
	\begin{align}\label{eq;asymptotic2}
	A( x , \bv)\ge A( x , \bu _{i }).
	\end{align}
	The second inequality of (\ref{eq;best bound})
	follows from (\ref{eq;asymptotic1}) and 
	(\ref{eq;asymptotic2}). 
	
	Let $\bu _{i+j}=(p_1, p_2, q)$, $\bu _i=(s_1,s_2, l)$ and choose 
	$k\in \{1, 2\}$
	such that $A(\widehat \bu _{i+j}, \bu _i)=|lp_k/q-s_k|^{1/w_k}$.
	We have 
	\begin{align}
	A( x , \bu _i) ^{w_k} & \ge |lx_k-s_k| \notag \\
	&= l|x_k-s_k/l|\notag \\
	&\ge l(|p_k/q-s_k/l|-|x_k-p_k/q|) \notag \\
	& = |lp_k/q-s_k|-(l/q)|x_kq-p_k| \notag \\
	& \ge A(\widehat{\bu }_{i+j}, \bu _i)^{w_k}-A( x , \bu _{i+j})^{w_k} \notag \\
	& > A(\widehat\bu _{i+j}, \bu _i)^{w_k}-A( x , \bu _i)^{w_k}. \notag                   
	\end{align}
	Therefore $A( x , \bu _i)> {2^{-1/w_2}}A(\widehat\bu _{i+j}, \bu _i)$. 
\end{proof}

Note that $r(\bu_{i+1})\le A(\widehat \bu_{i+1}, \bu_i)$. So Lemma \ref{lem;bound sequence} implies
the following corollary. 
\begin{cor}\label{cor;same size}
		Let $x\in \R^2\setminus \Q^2$ and  $\Sigma_x=\{ \bu _i\}_{i\in \N}$. Then for all $i\in \N$
		\[
		A(x, \bu_i)\asymp r(\bu_{i+1})\asymp A(\widehat \bu_{i+1}, \bu_i)
		\]
		where the implied constants do not depend on $i$. 
\end{cor}	

The following lemma gives a description of a $w$-singular vector using
its associated  best approximation sequence and it follows directly  from the definition. 
\begin{lem}\label{lem;charcaterise}
	Let $x\in \R^2\setminus \Q^2$ and  $\Sigma_x=\{ \bu _i\}_{i\in\N}$. Then $ x \in \mathrm{Sing}(w)$  if and only if 
	\[
	A(x, \bu _i) |\bu _{i+1}|\to 0 \quad \mbox{as }i\to \infty. 
	\]
\end{lem}

In view of Corollary \ref{cor;same size} and Lemma \ref{lem;charcaterise} one has the following corollary. 
\begin{cor}\label{cor;singular}
	Let $x\in \R^2\setminus \Q^2$ and  $\Sigma_x=\{ \bu _i\}_{i\in \N }$. Then $ x \in \mathrm{Sing}(w)$  if and only if 
	\[
	r(\bu_i)|\bu _{i}|\to 0 .
	\]
\end{cor}

In view of Corollary \ref{cor;singular}, for $\varepsilon>0$ the  set  
\begin{align}
Q_{\varepsilon }&=\{\bu \in Q: r(\bu ) |\bu |<\varepsilon , |\bu|>1 \}       \notag
\end{align}
consists of  best approximates of   almost (or  $\varepsilon$-close) $w$-singular vectors. 
We are going to define a relation $\sigma_\varepsilon$ on $Q_\varepsilon$ so that together with  the 
map $\beta$ on $Q$ defined by 
\begin{align}\label{eq;rectangle}
\beta(\bu )=\overline {B_w(\bu, |\bu|^{-1})}=I(\widehat \bu ; |\bu |^{-(w_1+1) }   , |\bu |^{-(w_2+1)})
\end{align}
we get an admissible  fractal relation $(Q_\varepsilon, \sigma_\varepsilon, \beta)$ such that the corresponding  fractal 
contains $\mathrm{Sing}(w) ^*$.

Now we  fix  $\bu \in Q_\varepsilon $ and 
 define the set  $\sigma_\varepsilon (\bu)$. 
We
choose $\bu'\in Q$ with the property   $|\bu'|\le |\bu |/2$ and 
$r(\bu )=A(\widehat\bu ,\bu' )$.
Let 
 $H_\bu $ be the hyperplane in $\R^3$ generated by $\bu$ and 
  $\bu'$. 
Let 
\[
D(\bu,\varepsilon ) =\{\bv\in H_\bu\cap Q_\varepsilon:  |\bv|  \ge  |\bu| ,  A(\widehat \bv, \bu )<  2^{2/w_2}r(\bu )\}.
\]
We note here that all the $\bv\in H_\bu\cap Q_\varepsilon$ (including $\bu $) with 
$\widehat \bv$  close to $\widehat \bu $ belong to $D(\bu, \varepsilon )$.
For every  $\bv\in D(\bu , \varepsilon)$ let
  \[
  E(\bu , \bv, \varepsilon)=\{  \bw\in Q_\varepsilon: |\bw|>|\bv|, \bw\not\in H_\bu\mbox{ and }   A(\widehat \bw , \bv)<{\varepsilon }|\bw |^{-1}\}.
  \]  
We define
\[
\sigma_{\varepsilon }(\bu )=\bigcup _{\bv\in D(\bu, \varepsilon )} E(\bu , \bv, \varepsilon).  
\]

\begin{lem}\label{lem;cover}
	For every  $0<\varepsilon<1$  one has 	$\mathrm{Sing}(w)^*\subset \mathcal F (Q_{\varepsilon }, \sigma_\varepsilon, \beta)$. 
\end{lem}
\begin{proof}
	Let $x\in \mathrm{Sing}(w)^*$ and  $\Sigma_x=\{ \bu _i\}_{i\in \N}$.
	We are going to construct a subsequence $\{\bu_{i_n} \}_{n\in \N }$ such that $(\bu_{i_n}, \bu_{i_{n+1}})\in \sigma_\varepsilon$
	and $x\in \beta (\bu_{i_n})$ for all  $n\in\N$. This will complete the proof.

	According to  Corollary \ref{cor;singular} there exists $i_0\in\N$ such that for  $i\ge i_0$
	one has
	\begin{align}\label{eq;raining}
	r(\bu _i)<\frac{\varepsilon 2^{-2/w_2}}{|\bu _i|}.
	\end{align}
	By  Lemma \ref{lem;bound sequence} 
	\[
	A( x , \bu _i)\le 2^{1/w_2} r(\bu _{i+1})< \frac{\varepsilon 2^{-1/w_2}}{|\bu _{i+1}|}<
	\frac{1}{|\bu _i|},
	\]
	which implies that $ x \in B_w(\bu _i, |\bu_i|^{-1})$. 
	Let $i_1=i_0$ and we inductively define $i_{n+1}$ to  be  smallest integer of 
	\[
	\{m\in \N : m>i_n,  \bu _m\not \in H_{\bu _{i_n}}\}
	\]
	which is nonempty since  $1, x_1, x_2$ are linearly independent over $\Q$.

	To simplify the notation we take  $\bu =\bu _{i_n},\bv=\bu_{i_{n+1}-1}$ and $\bw =\bu _{i_{n+1}}$. 
	It suffices to show $\bv\in D(\bu ,\varepsilon)$ and  $\bw\in E(\bu, \bv, \varepsilon)$. 
	Using  (\ref{eq;best bound}) we have
	\begin{align}
	 A(\widehat \bv, \bu )&< 2^{1/w_2} A( x , \bu ) \notag\\
	&\le  2^{1/w_2} A( x , \bu _{i_n-1}) \notag \\
	&< 2^{2/w_2}  r(\bu ), \notag            
	\end{align}
	which implies $\bv\in D(\bu, \varepsilon )$. 
	Using (\ref{eq;best bound}) again and (\ref{eq;raining})  we have
	\[
	A(\widehat \bw , \bv)< 2^{2/w_2}r(\bw )< \varepsilon/|\bw |.
	\]
	Therefore $\bw \in E(\bu, \bv, \varepsilon)$. 
\end{proof}

\subsection{Upper bound}\label{sec;upper}
For every $\varepsilon>0$ sufficiently small  we estimate 
 the Hausdorff dimension of $\mathcal F(Q_\varepsilon, \sigma_\varepsilon, \beta)$.
 In view of Lemma \ref{lem;cover} and the discussion at the beginning of
 \S \ref{sec;upper best} this will give an upper bound of the Hausdorff dimension of $\mathrm{Sing}(w)^*$
 and $\mathrm{Sing}(w)$. 
 
 	For  $\bu=(p, q)\in Q$ we let  $\pi_\bu : \R^3 \to \R^2$ be the projection along
 	$\R \bu $ defined by 
 	\begin{align}\label{eq;pi v}
 	\pi_\bu (( x , x_3))=  x -\frac{x_3}{q} p .
 	\end{align}  
 	The  kernel of $\pi_\bu $ is $\R\bu  $ and $\R\bu \cap \Z^3=\Z\bu $. 
 	The set 
 	\begin{align}\label{eq;lambda v}
 	\Lambda_\bu:=\pi_\bu (\Z^3)
 	\end{align}
 	 is a lattice in $\R^2$ with covolume $1/|\bu |$. 
 	 It is easy to see that $A(\widehat \bu , \bv)=\|\pi_\bu (\bv)\|_w$
 	 for every $\bv\in Q$. Therefore 
 	 \begin{align}\label{eq;rv min}
 	 r(\bu)=\inf_{y\in \Lambda_\bu\setminus \{0 \}} \|y\|_w. 
 	 \end{align}
 
\begin{lem}\label{lemma-D}
	Let $0<\varepsilon< 1$ and  $\bu\in Q_\varepsilon$. 
	For every real number $t$ with  $2<t\le 3$ we have
	\begin{align}\label{eq;lemma-D}
	\sum_{\bv\in D\left(\bu, \varepsilon \right)}\left(\frac{\left|\bu\right|}{\left|\bv\right|}\right)^{t}\ll\frac{1}{t-2}.
	\end{align}
\end{lem}
\begin{proof}
	For each $k\in \N$, let 
	\[
	D_{k}=\left\{ \bv\in D\left(\bu, \varepsilon\right)\::\: k\left|\bu\right|\leq\left|\bv\right|<\left(k+1\right)\left|\bu\right|\right\} .
	\]
	Since 
	\[
	\sum_{\bv\in D(\bu, \varepsilon )}\left(\frac{\left|\bu\right|}{\left|\bv\right|}\right)^{t}=\sum_{k=1}^{\infty}\sum_{\bv\in D_k}\left(\frac{\left|\bu\right|}{\left|\bv\right|}\right)^{t}\leq\sum_{k=1}^{\infty}\frac{\sharp D_{k}}{k^{t}},
	\]
	it  suffices  to show  $\sharp D_{k}\ll k$.

	Let $\pi_\bu$ and $\Lambda_\bu$ be as in (\ref{eq;pi v}) and (\ref{eq;lambda v}) respectively.  
	Since $\bu$ is a primitive vector of $\Z ^{3}$, the projection
	$\pi_{\bu}$ induces a bijection between 
	\[
	\left\{ \bv\in\mathbb{Z}^{3}\::\: k\left|\bu\right|\leq |\bv|<\left(k+1\right)\left|\bu\right|\right\} 
	\]
	and  $\Lambda_\bu$
	(note that  $\pi_{\bu}\left(\bv\right)=\pi_{\bu}\left(\bw\right)$ implies 
	 $\left|\bv\right|\equiv\left|\bw\right|\mbox{mod}\left|\bu\right|$).
	It follows that $\pi_{\bu}$ induces a bijection between 
	\[
	\left\{ \bv\in H_{\bu}\cap\Z^{3}\::\: k\left|\bu\right|\leq |\bv|<\left(k+1\right)\left|\bu\right|\right\} 
	\]
	and $\Lambda_{\bu}^{\prime}:=\Lambda_{\bu}\cap\pi_{\bu}\left(H_{\bu}\right)$.
	According to the  definition of $D(\bu, \varepsilon)$ and (\ref{eq;rv min})
	there exists  $ x\in \Lambda_{\bu}$ such
	that $\|{x }\|_w=r\left(\bu\right)$ and  $\Lambda_\bu'=\Z x $.
	Let $\bv\in D_k$, then $\pi_\bu(\bv)=s x$ for some $s\in \Z$. 
	Let $i\in \{1, 2 \}$ such that $\|x\|_w=|x_i|^{1/w_i}$. Then 
\begin{align*}
A\left(\widehat{\bv},\bu\right)&=\| |\bu|\widehat \bv-|\bu |\widehat \bu  \|_w=
\| |\bu| |\bv|^{-1} \cdot  (|\bv| \widehat\bv  -|\bv|\widehat \bu)   \|_w \\
&= \left\| {|\bu|}{|\bv|^{-1}} s x \right\|_w \ge\left ( \frac{|s|}{k+1}\right)^{1/w_i} r(\bu).\end{align*}
On the other hand, we have $A(\widehat \bv, \bu )< 2^{2/w_2}r(\bu)$ since $\bv\in D(\bu, \varepsilon)$.
It follows  that $|s|\ll k$ and hence $\sharp D_k\ll k$. 

%
%
\end{proof}
\begin{lem}\label{lem;equality}
	Let  $0<\varepsilon\le  2^{-2/w_2}$.
	For all   
	 $\bu\in Q_\varepsilon$ and $\bv\in D(\bu, \varepsilon)$
	 one has $\pi_\bu(H_\bu)=\pi_\bv(H_\bv)$  and $H_\bu=H_\bv$.  
\end{lem}
\begin{proof}
		Since $\bv\in H_\bu$,   $\pi_{\bu}(H_\bu)=\pi_{\bv}(H_\bv)$ implies 
		$H_\bu=H_\bv$. So  it suffices to prove the former. 
	According to (\ref{eq;rv min}) and the assumption $\bu \in Q_\varepsilon$, 
	there exists   $y\in \Lambda_\bu\cap \pi_{\bu}(H_\bu)$ such that 
	\[
	\|y\|_w=r(\bu)\le \frac{\varepsilon}{|\bu|}.
	\]	
	It follows that
	\[
	\lambda_1(K, \Lambda_\bu)\le \varepsilon^{w_2}
	\quad
	\mbox{where} \quad K=\{x\in \R^2: \|x\|_w\le |\bu|^{-1} \}.  
	\]  
		Since $\cov (\Lambda_\bu)=\frac{1}{4}\vol(K)=1/|\bu|$, Minkowski's second theorem (see (\ref{eq;minkowski})) implies that  
		\[\lambda_2(K, \Lambda_\bu)\ge {2^2 \over 2!} \cdot \lambda_1^{-1}\cdot {{\rm cov} (\Lambda_\bu) \over {\rm vol}(K)}\ge 2^{-1}\varepsilon^{-w_2}.\]
	Similarly, there exists 
	$z\in \Lambda_\bv\cap \pi_{\bv}(H_\bv)$ such that 
	\begin{align}\label{eq;starbuck}
	\|z\|_w=r(\bv)\le \frac{\varepsilon}{|\bv|}\le  \frac{\varepsilon}{|\bu|}.
	\end{align}
	Since $\bv\in H_\bu$ one has $\pi_\bu(\bx)-\pi_\bv(\bx)\in \R y$
	for all $\bx\in \R^3$.
		It follows that $\Lambda_\bv\subset \Lambda_\bu+\R y$, and hence  
 $z\in \Lambda_\bu+\R y$.
 We will show that  $z\in \R y$, 
 which  will imply  $\pi_{\bu}(H_\bu)=\pi_{\bv}(H_\bv)$ and   complete the proof.

Suppose $z\not \in \R y$, then there exists 
   $0\le s<1$ such that $z+sy\in \Lambda_\bu \setminus \R y$. 
    So 
   \begin{align*}
   \|z\|_K\ge \|z+sy\|_K-\|y\|_K \ge {2^{-1}}{\varepsilon^{-w_2}}- \varepsilon^{w_2} \ge 1
   \end{align*}
   where $\|\cdot\|_K$ be the norm on $\R^2$ defined in (\ref{eq;knorm}).   
    This contradicts (\ref{eq;starbuck}). Therefore $z\in \R y$.
\end{proof}

\begin{lem}\label{lemma-E}
	Let $0< \varepsilon \le 2^{-2/w_2}$.
	For all $\bu\in Q_\varepsilon$,   $\bv\in D(\bu, \varepsilon)$
	and  $2<t\le 3$ one has 
	\begin{align}\label{eq;sofa}
	\sum_{\bw\in E\left(\bu,\bv,\varepsilon\right)}\left(\frac{\left|\bv\right|}{\left|\bw\right|}\right)^{t}\ll\frac{\varepsilon}{t-2}.
	\end{align}
\end{lem}
\begin{proof}

Let $\Lambda_\bv=\pi_\bv(\Z^3)$ and  $\Lambda_\bv^\circ=\Lambda_\bv\setminus \pi_\bv(H_\bu)$.
Since $H_\bu=H_\bv$ according to  Lemma \ref{lem;equality}, we have 
$\Lambda_\bv^\circ=\Lambda_\bv\setminus \pi_\bv(H_\bv)$.
 Let $y\in \Lambda_\bv\cap\pi_\bv( H_\bv)$ with $\|y\|_w=r(\bv)\le \veps/|\bv|$.
For each $k\in \N$, let 
\[
E_{k}=\left\{ \bw\in E\left(\bu,\bv,\varepsilon\right)\::\: k\left|\bv\right|\leq\left|\bw\right|<\left(k+1\right)\left|\bv\right|\right\} .
\]
Since
\[
\sum_{\bw\in E\left(\bu,\bv,\varepsilon\right)}\left(\frac{\left|\bv\right|}{\left|\bw\right|}\right)^{t}=\sum_{k=1}^{\infty}\sum_{\bw\in E_k}\left(\frac{\left|\bv\right|}{\left|\bw\right|}\right)^{t}\leq\sum_{k=1}^{\infty}\frac{\sharp E_k}{k^{t}},
\]
it  suffices  to show that $\sharp E_{k}\ll \veps k$. 
%
%
Let $\bw\in E_k$ and write $z_\bw=\pi_\bv(\bw)$. 
Then $\bw=\left(z_\bw+{|\bw|}{\widehat \bv},|\bw|\right)$ and hence 
by the definition of $E(\bu, \bv, \varepsilon)$
\begin{align}\label{eq;A of proj}
A(\widehat \bw,\bv)=\left\|{|\bv|}{|\bw|^{-1}}z_\bw\right\|_w\le \veps|\bw|^{-1}.
\end{align}

Consider the convex set \[M_k=\left \{x\in \R^2:\left \|\frac{1}{k+1}x\right \|_w\le\frac{\veps }{k|\bv|}\right \}.\]
In view of (\ref{eq;A of proj}) and the  definition of $E_k$ we have the inclusion $\pi_\bv(E_k)\subseteq \Lambda_{\bv}\cap M_k$. So $\sharp E_k\le \sharp M_k\cap \Lambda_\bv$ since $\pi_\bv|_{E_k}$ is injective.

Note that $y\in \Lambda_\bv $ is always in $M_k$. 
So if  $E_k$ is nonempty we have $\lambda_2(M_k,\Lambda_{\bv})\le 1$. 
  Hence by Lemma \ref{cor;count 1} we get 
 \[\sharp M_k\cap\Lambda_\bv\ll \frac{\vol\, M_k}{\cov\, \Lambda_\bv}. \]
Note  that $\vol\, M_k=\frac{4\veps}{k|\bv|}(k+1)^2$ and   $\cov\, \Lambda_\bv=\frac{1}{|\bv|}$.  Therefore  
$\sharp E_k\le\sharp M_k\cap\Lambda_\bv\ll \veps k$ as desired.
\end{proof}


Now we estimate the upper bound of the Hausdorff dimension of  $\mathrm{Sing}(w)^*$. 
\begin{thm}\label{thm;estimate} 
	There exists $C> 0$ such that for all   $0< \varepsilon \le 2^{-2/w_2}$  the
	 Hausdorff dimension of $\mathcal F(Q_\varepsilon, \sigma_\varepsilon, \beta)$
	 is less  than  or equal to
	 \begin{align}\label{eq;upper bound}
	 2-\frac{1}{1+w_1}+C\sqrt{\veps}. 
	 \end{align}
	 Therefore 
	the Hausdorff dimension of  $\mathrm{Sing}(w)^*$ and $\mathrm{Sing}(w)$ is less   than  or equal to 
	$
	2-\frac{1}{1+w_1}.
	$
	\end{thm}
\begin{proof}
	Let $C' $ be the product of implied constants of (\ref{eq;lemma-D}) and (\ref{eq;sofa}). 
	By Lemma \ref{lem;cover} and the discussion at the beginning of \S \ref{sec;upper best} it suffices 
	to 
	show that   the Hausdorff dimension of $\mathcal F(Q_\varepsilon, \sigma_\varepsilon, \beta)$
	is less than or equal to (\ref{eq;upper bound}) for $C= {1 \over 1+w_1} \sqrt{C'}$.
    In view of  Lemma \ref{thm-general-upper} it suffices to show that for all 
	\begin{align}\label{eq;finish}
	s>2-\frac{1}{1+w_1}+ C \sqrt{\varepsilon}
	\end{align}
	one has 
	\begin{align}\label{eq;estimate 1}
	\sum_{ \bw \in \sigma_{\varepsilon}(\bu)} L( \bw )\cdot W( \bw )^{s-1} \leq L(\bu)\cdot W(\bu)^{s-1}
	\end{align}
	where $L(\bw)=2|\bw|^{-w_2-1}$ and $W(\bw)=2|\bw|^{-w_1-1}$. 
	
    Note that (\ref{eq;estimate 1}) is equivalent to 
	\begin{eqnarray}\label{eq;finish 1}
		\sum_{ \bw \in\sigma_{\varepsilon}\left(\bu\right)}\left(\frac{\left|\bu\right|}{\left| \bw \right|}\right)
		^{(s-1)(w_1+1)+w_2+1} \leq 1.
	\end{eqnarray}
	By Lemmas \ref{lemma-D} and \ref{lemma-E},  for all $t>2$ one has 
	\begin{eqnarray*}
	\sum_{ \bw \in\sigma_{\varepsilon}\left(\bu\right)}\left(\frac{\left|\bu\right|}{\left| \bw \right|}\right)^{t}\leq\sum_{\bv\in D\left(\bu, \varepsilon\right)}\left(\frac{\left|\bu\right|}{\left|\bv\right|}\right)^{t}\sum_{ \bw \in E\left(\bu,\bv,\varepsilon\right)}\left(\frac{\left|\bv\right|}{\left| \bw \right|}\right)^{t}\leq\frac{C' \cdot\varepsilon}{\left(t-2\right)^{2}}.
	\end{eqnarray*}
	Plugging in   $t=(s-1)(w_1+1)+w_2+1$ which is $>2+ \sqrt{C'\veps}$ by 
    (\ref{eq;finish}) in the  above inequality, we get 
	 (\ref{eq;finish 1}).
\end{proof}

\begin{proof}[Proof of Theorem \ref{thm;main}]
	The authentic weighted cases ($w_1>w_2$)
	 follow from Theorems \ref{thm;lower bound} and \ref{thm;estimate} and the unweighted case
	 ($w_1=w_2$) is proved in \cite{c11}. 
\end{proof}

\begin{proof}[Proof of Theorem \ref{thm;improve}]
	We will use notations of  \S \ref{sec;best}. 
	Let $\DI (w, \varepsilon )^*$ be the set of $x\in \DI (w, \varepsilon )$ such that 
	$1, x_1, x_2$ are linearly independent over $\Q$. 
	Let  $x\in \DI (w, \varepsilon )^*$ and let $\Sigma_x=\{\bu_i \}_{i\in \N}$ be the 
	fixed sequence of $w$-best approximates of $x$. 
	It follows from definition that there exists  $i_1\in \N$ such that for $i\ge i_1$ one has 
	\begin{align}\label{eq;shoe}
	A(x, \bu_i)< \frac{\varepsilon}{|\bu_{i+1}|}. 
	\end{align}
	On the other hand the first inequality of (\ref{eq;best bound}) implies that for all $i\in\N$
	\begin{align}\label{eq;shoe 1}
	2^{-1/w_2}r(\bu_{i+1})<A(x, \bu_i). 
	\end{align}
	It follows form (\ref{eq;shoe}) and (\ref{eq;shoe 1}) that  for all $i\ge i_1$
	\[
	r(\bu_{i+1}) < \frac{2^{1/w_2}\varepsilon}{|\bu_{i+1}|}. 
	\]
    Note that in the proof of 
	Lemma \ref{lem;cover}, we only use (\ref{eq;raining}) and the fact that    $1, x_1, x_2$ are linearly independent over $\Q$. Therefore the same  argument implies 
	\[
	x\in \mathcal F(Q_{\varepsilon2^{3/w_2}}, \sigma_{\varepsilon2^{3/w_2}}, \beta).
	\] 
	So we have 
	 \[
	 \DI(w, \varepsilon)^*\subset \mathcal F(Q_{\varepsilon 2^{3/w_2}}, \sigma_{\varepsilon 2^{3/w_2}}, \beta). 
	 \]
 By Theorem \ref{thm;estimate}
	 \begin{align}\label{eq;additional}
	 \dim_H  \DI(w, \varepsilon)^* \le 2-\frac{1}{1+w_1}+C\sqrt \varepsilon
	 \end{align}
	 where the constant $C$ 
	 is  independent of $\varepsilon$. 
	 The conclusion of Theorem \ref{thm;improve} follows from (\ref{eq;additional})
     and the observation that 	 
	   $\DI(w, \varepsilon)\setminus \DI(w, \varepsilon)^* $ is contained in 
	 a countable union of lines in $\R^2$. 
	
\end{proof}

\textbf{Acknowledgements:} We would like to thank Barak Weiss and the referee  for carefully reading the draft  and 
helping us improve the paper.


\begin{thebibliography}{99}
	
\bibitem{a1}
J. An, {\em  Badziahin-Pollington-Velani's theorem and Schmidt's game}, Bull. Lond. Math. Soc. {\bf 45} (2013),
	no. 4, 721-733.
	 
\bibitem{a2} J. An, 	{\em Two-dimensional badly approximable vectors and Schmidt's game}, Duke Math. J.
{\bf 165} (2016), no. 2, 267-284. 
	 

\bibitem{bpv}
 D. Badziahin, A. Pollington, S. Velani, {\em On a problem in simultaneous Diophantine approximation:
Schmidt's conjecture}, Ann. of Math. (2) {\bf 174} (2011), no. 3, 1837-1883.

\bibitem{b15}
V. Beresnevich, {\em Badly approximable points on manifolds},  Invent. Math. {\bf 202} (2015), no. 3, 1199-1240.
	 




	
	\bibitem{cassels}
	J.W.S. Cassels,   \emph{An Introduction to the Geometry of Numbers}, Springer 1959.
		

\bibitem{c11}
Y. Cheung, {\em Hausdorff dimension of the set of singular pairs}, Ann. of Math. {\bf 173} (2011), 127-167.

\bibitem{cc}
Y. Cheung and N. Chevallier, {\em Hausdorff dimension of singular vectors}, 
Duke Math. J., {\bf165}, (2016), 2273-2329..

\bibitem{d85}
S. G. Dani, {\em
Divergent trajectories of flows on homogeneous spaces and Diophantine approximation},  J. Reine Angew. Math. {\bf 359} (1985), 55-89. 

\bibitem{ek12}
M. Einsiedler and  S, Kadyrov, {\em Entropy and escape of mass for $SL_3(\Z)\backslash SL_3(\R)$}, 
 Israel J. Math. {\bf 190} (2012), 253-288.
 
\bibitem{emm98}
A. Eskin, G. A. Margulis, and S. Mozes, {\em Upper bounds and asymptotics in a
	quantitative version of the Oppenheim conjecture}, Ann. of Math. {\bf 147} (1998),
93--141.

\bibitem{kklm}
S. Kadyrov, D.  Kleinbock, E.  Lindenstrauss, and G. A.  Margulis, {\em Singular systems of linear forms and non-escape of mass in the space of lattices}. J. Anal. Math. {\bf 133} (2017), 253-277.

\bibitem{schmidt}
W.M. Schmidt, {\em Diophantine approximation},  Lecture Notes in Mathematics {\bf 785}, Springer, Berlin, 1980.

\bibitem{siegel}
C. L.Siegel, {\em Lectures on the geometry of numbers}, Springer-Verlag, Berlin, 1989. 









\end{thebibliography}
\end{document}